\documentclass[final,leqno]{siamltex}

\usepackage{amsmath}
\usepackage{float}
\usepackage{amsfonts}
\usepackage{amssymb}
\usepackage{verbatim}
\usepackage{enumerate}
\usepackage[square,sort,comma,numbers]{natbib}
\usepackage{multirow}
\usepackage{hhline}
\usepackage{array}
\usepackage{caption}
\usepackage{graphicx}
\usepackage{color}

\newcommand{\neigh}{\mathcal{N}}
\newcommand{\coord}{{J}}

\newcommand{\rset}{\mathbb{R}}

\newcommand{\average}{\mathbb{E}}
\newcommand{\bo}{\mathbf}
\newcommand{\prox}{\text{prox}_{\Psi}}

\newcommand{\bl}{\color{black}}

\newtheorem{assumption}{Assumption}

\newtheorem{remark}{Remark}
\newenvironment{example}[1][{\normalfont\textit{Example.\theproposition.}}]
{\begin{trivlist}\item[\hskip \labelsep {\bfseries #1}]}{\hfill$\square$\end{trivlist}}

\usecounter{algorithmenumi}

\title{ Parallel random coordinate descent method for
composite minimization: convergence analysis and  error bounds}

\author{ Ion Necoara, Dragos  Clipici
\thanks{I. Necoara and  D. Clipici are with  University  Politehnica Bucharest,
Automatic Control and Systems Engineering Department, 060042
Bucharest, Romania. {\tt\small \{ion.necoara,
dragos.clipici\}@acse.pub.ro}.}}

\date{\normalsize \textbf{November 2013}}

\begin{document}
\maketitle

\begin{abstract}
In this paper we {\bl consider} a parallel version of  a randomized
 (block) coordinate descent method for minimizing the sum of a
partially separable smooth convex  function and a fully separable
non-smooth convex function. Under the assumption of  Lipschitz
continuity of the gradient of the  smooth function, {\bl this method
has} a sublinear convergence rate. Linear convergence rate of the
method is obtained for the newly introduced class of
\textit{generalized error bound functions}. We prove that the new
class of generalized error bound functions encompasses both
global/local error bound functions and smooth strongly convex
functions.  We also show that the theoretical estimates on the
convergence rate depend on the number of blocks chosen randomly and
a natural measure  of  separability of the objective function.
Numerical simulations are also provided to confirm our theory.
\end{abstract}


\pagestyle{myheadings} \thispagestyle{plain} \markboth{I. Necoara
and   D. Clipici}{Parallel coordinate descent methods for composite
minimization}


\section{Introduction}
In recent years there has been an ever-increasing interest  in the
optimization community for algorithms suitable for solving convex
optimization problems with a very large number of variables. These
problems, known as big data problems, have arisen from more recent
fields such as network control \cite{NecClip:13a,NecSuy:09}, machine
learning \cite{Bis:06} and data mining \cite{WitFra:06}. An
important property of these problems  is that they are
\textit{partially separable}, which permits  parallel and/or
distributed computations in the optimization algorithms that are to
be designed for them {\bl \cite{NecClip:13a,RicTak:13}}. This,
together with the surge of multi-core machines or clustered parallel
computing technology in the past decade, has led to the widespread
focus on coordinate descent~methods.

\noindent   \textbf{State of the art}:  Coordinate descent methods
are  methods in which a number of (block) coordinates updates of
vector of variables are conducted at each iteration. The reasoning
behind this is that coordinate updates for problems with a large
number of variables are much simpler than computing a full update,
requiring less memory and computational power,  and that they can be
done independently, making coordinate descent methods more scalable
and suitable for distributed and parallel computing hardware.
Coordinate descent methods can be divided into  two main categories:
deterministic and random. In deterministic coordinate descent
methods, the (block) coordinates which are to be updated at each
iteration are chosen in a  cyclic fashion or based on some greedy
strategy. For cyclic coordinate search,  estimates on the rate of
convergence were given recently in \cite{BecTet:13, HonWan:13},
while for the greedy coordinate search the convergence rate is given
e.g. in \cite{TseYun:09, Tse:01}. On the other hand, in random
coordinate descent methods, the (block) coordinates which are to be
updated are chosen randomly based on some probability distribution.
In \cite{Nes:12}, Nesterov presents a random coordinate descent
method for smooth convex problems, in which only one coordinate is
updated at each iteration. Under some assumption  of Lipschitz
gradient and strong convexity of the objective function, the
algorithm in \cite{Nes:12} was proved to have linear convergence in
the expected values of the objective function. In
\cite{Nec:13,NecNes:11} a 2-block random coordinate descent method
is proposed to solve linearly constrained smooth  convex problems.
The algorithm from  \cite{Nec:13,NecNes:11} was extended to linearly
constrained composite convex minimization in \cite{NecPat:12}.   The
results in \cite{Nes:12} and \cite{Nes:13} were combined  in
\cite{RicTak:12,LuXia:13}, in which the authors propose a randomized
coordinate descent method to solve  composite convex problems. To
our knowledge, the first results on the  linear convergence of
coordinate descent methods under more relaxed assumptions than
smoothness and strong convexity were obtained e.g. in
\cite{TseYun:09, LuoTse:93}. In particular, linear convergence of
these methods is proved under some local error bound property, which
is more general than the assumption of Lipschitz gradient and strong
convexity as required in
\cite{Nes:12,Nec:13,NecNes:11,NecPat:12,RicTak:12}. However, the
authors in \cite{TseYun:09, LuoTse:93} were able to show linear
convergence only locally. Finally, very few results were known in
the literature on distributed and parallel implementations of
coordinate descent methods. {\bl Recently, a more thorough
investigation regarding the separability of the objective function
and  ways in which the convergence can be accelerated through
parallelization was undertaken   in
\cite{NecNes:11,NecClip:13a,RicTak:12a,PenYan:13}, where it is shown
that speedup can be achieved through this approach}. Several other
papers on parallel coordinate descent methods have appeared around
the time this paper was finalized
\cite{BraKyr:11,Nec:13,RicTak:13,TakBij:13}. 

\noindent \textbf{Motivation}:  Despite widespread use of coordinate
descent  methods for solving large convex problems, there are some
aspects that have not been fully studied. In particular, in
practical applications, the assumption  of Lipschitz gradient and
strong convexity is  restrictive and the main interest is in finding
larger classes of functions for which we can still prove linear
convergence. We are also interested in providing schemes based on
parallel and/or distributed computations. Finally, the convergence
analysis has been almost exclusively limited to centralized stepsize
rules and local convergence results.  These represent the main
issues that we pursue here.

\noindent  \textbf{Contribution}: In this paper {\bl we consider a
parallel version of  random (block) coordinate gradient descent
method \cite{Nes:12,NecPat:12,RicTak:12}} for solving large
optimization problems with a convex separable composite objective
function, i.e. consisting of the sum of a partially separable smooth
function and a fully separable non-smooth function. Our approach
allows us to analyze in the same framework several methods: full
gradient, serial random coordinate descent and any parallel random
coordinate descent method in between. {\bl Analysis of coordinate
descent methods based on updating in parallel  more than one (block)
component per iteration was given first in
\cite{NecNes:11,BraKyr:11,NecClip:13a} and then further studied e.g.
in  \cite{RicTak:12a,PenYan:13}.  We provide a  detailed rate
analysis for our parallel coordinate  descent algorithm under
general assumptions on the objective function, e.g. error bound type
property,  and under more knowledge on the structure of the problem
and we prove substantial improvement on the convergence rate w.r.t.
the existing results from the literature.}. In particular, we show
that this algorithm attains linear convergence for problems
belonging to a general class, named {\em generalized error bound
problems}. We establish that our class includes  problems with
global/local error bound objective functions and implicitly strongly
convex functions with some Lipschitz continuity property on the
gradient. We also show that the new class of problems that we define
in this paper covers many  applications in networks. Finally, we
establish that the theoretical estimates on the convergence rate
depend on the number of blocks chosen randomly and a natural measure
of separability of the objective function. In summary, the
contributions of this paper~include:


\noindent \textit{(i)} {\bl We employ a parallel version of
coordinate descent method   and show that the algorithm  has
sublinear convergence rate in the expected values of the objective
function}. For this algorithm, the iterate updates can be done
independently and thus it is suitable for  parallel and/or
distributed computing~architectures.

\noindent \textit{(ii)} We introduce a new class of  generalized
error bound problems  for which we show that it encompasses both,
problems with global/local error bound functions and smooth strongly
convex functions and that it covers many practical applications.

\noindent \textit{(iii)} Under the generalized error bound property,
we prove that our parallel random coordinate descent algorithm has a
global linear convergence rate.

\noindent \textit{(iv)} We also perform a theoretical
identification of which categories of problems and objective
functions satisfy the generalized error bound property.

\noindent \textbf{Paper Outline}: In Section \ref{sec_probdef} we
present our optimization model and discuss  practical applications
which can be posed in this framework. {\bl In Sections
\ref{sec_rcdm} and \ref{sec_smooth} we analyze the properties of a
parallel random coordinate descent algorithm, in particular  we
establish sublinear convergence rate, under some Lipschitz
continuity assumptions}. In Section \ref{sec_gebf} we introduce the
class of generalized error bound problems  and  we prove that {\bl
the random coordinate descent algorithm} has  global linear
convergence rate under this property. In Section
\ref{sec_gebfproperty} we investigate which classes of optimization
problems have an objective function that satisfies the generalized
error bound property. In Section \ref{sec_comparison} we discuss
implementation details of the algorithm  and compare it with other
existing methods. Finally, in Section 8 we present some preliminary
numerical tests on constrained lasso problem.


\section{Problem formulation}
\label{sec_probdef} In many big data applications arising from e.g.
networks,   control and data ranking, we have a system formed from
several entities, with a communication
 graph which indicates the interconnections between entities
(e.g.  sources and links in network optimization \cite{RamNed:09},
website pages in data ranking \cite{Bis:06} or subsystems in control
\cite{NecSuy:09}). We denote this \textit{bipartite graph} as
$G=([N] \times [\bar{N}],E)$, where $[N] =\left\{1,\dots,N\right\}$,
$[\bar{N}] =\left\{1,\dots,\bar{N}\right\}$ and $E \in
\left\{0,1\right\}^{{N} \times \bar{N}}$ is an incidence matrix. We
also introduce two sets of neighbors $\neigh_j$ and $\bar{\neigh}_i$
associated to the graph, defined as:
 \[ \neigh_j=\{ i \in [N]: \;  E_{ij} =1  \} \quad \forall j \in [\bar{N}]
~~ \text{and} ~~ \bar{\neigh}_i = \{j \in [\bar{N}]: \; E_{ij} =1 \}
\quad  \forall i \in [N]. \]

\noindent The index sets $\neigh_j$ and $\bar{\neigh}_i$, which e.g.
in the context of  network optimization  may represent  the set of
sources which share the link $j \in [\bar{N}]$ and  the set of links
which are used by the source $i \in [N]$,  respectively, describe
the local information flow in the graph. We denote the entire vector
of variables for the graph as $x \in \rset^n$. The vector $x$ can be
partitioned accordingly in block components $x_i \in \rset^{n_i}$,
with $n=\sum_{i=1}^N n_i$. In order to easily extract subcomponents
from the vector $x$, we consider a partition of the identity matrix
$I_n=\left [U_1 \dots U_N \right ]$, with $U_i \in \rset^{n \times
n_i}$, such that $x_i=U_i^T x$ and matrices $U_{\neigh_i}\in
\rset^{n \times n_{\neigh_i}}$, such that
$x_{\neigh_i}=U_{\neigh_i}^T x$, with $x_{\neigh_i}$ being the
vector containing all the components $x_j$ with $j \in \neigh_i$. In
this paper we address problems arising from such systems, where the
objective function can be written in a general form as (for similar
models  and settings, see {\bl
\cite{NecClip:13a,RicTak:12a,NiuRec:12,RamNed:09,TakBij:13,RicTak:13})}:

\begin{align}
 F^*=\min_{x \in \rset^n} F(x)  \quad  \left (= \sum_{j=1}^{\bar{N}} f_j
 (x_{\neigh_j}) + \sum_{i=1}^{N} \Psi_i(x_i) \right ), \label{gen_form}
\end{align}
where $f_j:\rset^{n_{\neigh_j}} \rightarrow \rset^{}$ and $\psi_i:
\rset^{n_i} \rightarrow \rset^{}$. We denote $f(x)=\sum_{j=1}^{\bar{N}} f_j
(x_{\neigh_j})$ and $\Psi(x)=\sum_{i=1}^N \Psi_i(x_i)$. The function
$f(x)$ is a smooth \textit{partially separable} convex function,
while $\Psi(x)$ is \textit{fully separable} convex  non-smooth function.
The local information structure imposed by the graph $G$ should be
considered as part of the problem formulation. We {\bl consider} the
following natural measure of separability of the objective function $F$:
\[ (\omega, \; \bar{\omega})
= (\max_{j \in [\bar{N}]} |\neigh_j |, \; \max_{i \in [N]} | \bar{\neigh}_i|).
  \]
Note that $1 \leq \omega \leq N$, $1 \leq  \bar{\omega} \leq
\bar{N}$ and the definition of the measure of separability $(\omega,
\; \bar{\omega})$ is more general than the one considered  in
\cite{RicTak:12a}  that is defined only in terms of $\omega$.  It is
important to note that coordinate gradient descent type methods for
solving problem \eqref{gen_form} are appropriate only in the case
when $ \bar{\omega}$ is relatively small, otherwise incremental type
methods \cite{RamNed:09,WanBer:13} should be considered for solving
\eqref{gen_form}. Indeed, difficulties may arise when $f$ is the sum
of a large number of component functions
 and $\bar{\omega}$ is  large, since in that case exact computation
of the components of   gradient (i.e. $\nabla_{i} f(x) = \sum_{j \in
\bar{\neigh}_i}  \nabla_i f_j(x_{\neigh_j})$) can be either very
expensive or impossible due to noise. In conclusion, {\bl we assume that the algorithm is employed for problems \eqref{gen_form},
 with $\bar{\omega}$ relatively small, i.e. $\bar{\omega}, \omega \ll n$}.

\noindent Throughout this paper, by $x^*$ we denote an optimal
solution of problem \eqref{gen_form} and by $X^*$ the set of optimal
solutions. We  define the index indicator function as:
$$\bo{1}_{\neigh_j}(i)=\begin{cases} 1,& \text{if} \; i \in \neigh_j \\ 0,
 &\text{otherwise}, \end{cases}$$
and the set indicator function as:
$$\bo{I}_X(x)=\begin{cases} 0,& \text{if} \; x \in X \\ +\infty,
 &\text{otherwise}. \end{cases}$$
 Also, by $\|\cdot\|$ we denote the standard Euclidean norm and
we introduce an additional norm $\|x\|_W^2=x^T Wx$, where $W \in \rset^{n \times n}$
 is a positive diagonal matrix.
  Considering these, we denote by $\Pi_X^W(x)$  the projection
of a point $x$ onto a set $X$ in the norm $\|\cdot\|_W$:
\begin{equation*}
\Pi_X^W(x)=\arg \min_{y \in X}  \| y-x\|_W^2.
\end{equation*}
Furthermore, for simplicity of exposition,  we denote by $\bar{x}$ the
projection of a point $x$ on the optimal set $X^*$, i.e.
$\bar{x}=\Pi_{X^*}^W(x)$. In this paper we consider that the smooth
component $f(x)$ of \eqref{gen_form} satisfies the following
assumption:

\begin{assumption}
\label{lip_fi}
We assume that the functions $f_j(x_{\neigh_j})$ have Lipschitz
continuous gradient with a constant $L_{\neigh_j}>0$:
\begin{equation}
\|\nabla f_j(x_{\neigh_j}) - \nabla f_j (y_{\neigh_j})\|  \leq
L_{\neigh_j} \|x_{\neigh_j} - y_{\neigh_jj}\| \quad \forall
x_{\neigh_j},y_{\neigh_j} \in \rset^{n_{\neigh_j}}.
\label{lipschitz_comp}
\end{equation}
\end{assumption}

\noindent Note that our assumption {\bl is  different from
the ones in} \cite{Nes:12,NecClip:13a,NecPat:12,RicTak:12a}, where the
authors consider that the gradient of the function $f$ is
coordinate-wise Lipschitz continuous, which states the following: if
we define the partial gradient $\nabla_i f(x)=U_i^T \nabla f(x)$,
then there exists some constants $L_i>0$ such that:
\begin{equation}
 \|\nabla_i f(x+U_i y_i) - \nabla_i f(x)\| \leq L_i
 \|y_i\| \quad \forall x \in \rset^n, \; y_i \in \rset^{n_i}.
 \label{lipschitz_grad}
\end{equation}
As a consequence of Assumption \ref{lip_fi} we have that
\cite{Nes:04}:
\begin{equation}\label{desc_fi}
 f_j(x_{\neigh_j} + y_{\neigh_j}) \leq f_j(x_{\neigh_j}) +
 \langle \nabla f_j(x_{\neigh_j}), y_{\neigh_j} \rangle +
 \frac{L_{\neigh_j}}{2} \|y_{\neigh_j}\|^2.
\end{equation}
Based on Assumption \ref{lip_fi} we can show  the following
distributed variant of the descent lemma, which is central in our
derivation of a parallel coordinate descent method and  in proving
the convergence rate for it.
\begin{lemma}
\label{lema_desc} Under Assumption \ref{lip_fi}  the following
inequality holds for the smooth part of the  objective function
$f(x)=\sum_{j=1}^N f_j(x_{\neigh_j})$:
\begin{equation}
 f(x+y)\leq f(x)+\langle \nabla f(x), y \rangle+
 \frac{1}{2}\|y\|_W^2\quad \forall x,y \in \rset^n
 \label{desc_lemma},
\end{equation}
where {\bl $W \succ 0$ is diagonal with its  blocks $W_{ii} \in
\rset^{n_i \times n_i}$, $W_{ii}= \sum\limits_{j \in
\bar{\neigh}_i}L_{\neigh_j} I_{n_i}, \;  i \in [N]$, and the
remaining blocks are zero}.
\end{lemma}

\begin{proof}
If we sum up \eqref{desc_fi} for $j \in [\bar{N}]$ and by the definition
of $f$ we have that:
\begin{equation}
f(x+y)\leq f(x)+\sum_{j=1}^{\bar{N}} \left [ \langle  \nabla
f_j(x_{\neigh_j}), y_{\neigh_j} \rangle + \frac{L_{\neigh_j}}{2}
\|y_{\neigh_j}\|^2 \right ]. \label{lip_basic}
\end{equation}
Given matrices $U_{\neigh_j}$, note that we can express the first
term in the right hand side as:
\begin{align*}
\sum_{j=1}^{\bar{N}} \left  \langle \nabla f_j(x_{\neigh_j}), y_{\neigh_j}
\right \rangle &\!=\!\sum_{j=1}^{\bar{N}}  \left \langle \nabla
f_j(x_{\neigh_j}), U_{\neigh_j}^T y \right \rangle=\sum_{j=1}^{\bar{N}}
\left \langle  U_{\neigh_j} \nabla f_j(x_{\neigh_j}), y \right
\rangle\!=\!\langle \nabla f(x), y\rangle.
\end{align*}
Note that since $W$ is a diagonal matrix we can express  the  norm
$\|\cdot\|_W$  as:
\begin{equation*}
\|y\|_W^2= \sum_{i=1}^N  \left (\sum_{j \in \bar{\neigh_i}}
L_{\neigh_j} \right ) \|y_i\|^2.
\end{equation*}

\noindent From the definition of $\neigh_j$ and $\bar{\neigh_i}$,
note that $\bo{1}_{{\neigh}_j}(i)$ is equivalent to
$\bo{1}_{\bar{\neigh}_i}(j)$. Thus, for the final term of the right
hand side of \eqref{lip_basic} we have that:
\begin{align*}
\frac{1}{2} \sum_{j=1}^{\bar{N}}  L_{\neigh_j} \|y_{\neigh_j}\|^2 &
=\frac{1}{2} \sum_{j=1}^{\bar{N}}  L_{\neigh_j} \sum_{i \in \neigh_j} \|y_i\|^2
= \frac{1}{2}  \sum_{j=1}^{\bar{N}} L_{\neigh_j} \sum_{i=1}^N \|y_i\|^2
\bo{1}_{\neigh_j}(i) \\
&=\frac{1}{2}  \sum_{i=1}^N \|y_i\|^2 \sum_{j=1}^{\bar{N}}
 L_{\neigh_j} \bo{1}_{\bar{\neigh}_i}(j) = \frac{1}{2} \sum_{i=1}^N \|y_i\|^2
\sum_{j \in \bar{\neigh_i}} L_{\neigh_j}=\frac{1}{2} \|y\|_W^2,
\end{align*}
and the proof is complete.
\end{proof}

\noindent {\bl Note that   the convergence results of this paper
hold for any descent lemma in the form   \eqref{desc_lemma} and thus
the expression of the  matrix $W$ above can be replaced with any
other   block-diagonal matrix $W \succ 0$ for which
\eqref{desc_lemma}  is valid.  Based on \eqref{lipschitz_grad} a
similar inequality as in \eqref{desc_lemma} can be derived, but the
matrix $W$ is replaced in this case with the matrix $\omega
W'=\omega \text{diag}(L_i I_{n_i}; i \in [N])$ \cite{RicTak:12a}.
These differences in the matrices will lead to different step sizes
in the algorithms of our paper and of e.g.
\cite{NecClip:13a,RicTak:12a}}.  Moreover, from the generalized
descent lemma through the norm $\|\cdot\|_W$, the sparsity induced
by the graph $G$ via the sets $\neigh_j$ and $\bar{\neigh}_i$ and
implicitly via the measure of separability $(\omega, \bar{\omega})$
will intervene in the estimates for the convergence rates of the
algorithm. A detailed discussion on this issue can be found in
Section \ref{sec_comparison}.   The following lemma establishes
Lipschitz continuity for $\nabla f$ but in the norm $\|\cdot\|_W$,
whose proof can be derived using similar arguments as in
\cite{Nes:04}:

\begin{lemma}
For a function  $f$ satisfying  Assumption \ref{lip_fi} the
following  holds:
\begin{equation}
 \|\nabla f(x)-\nabla f(y)\|_{W^{-1}} \leq \|x-y\|_W
 \quad  \forall x,y \in \rset^n. \label{gradient_lip}
\end{equation}
\end{lemma}

\subsection{Motivating practical applications}
\label{sub_sec_motiv} We now present  important applications from
which the interest for problems of type \eqref{gen_form} stems.

\noindent \textit{Application I}: One specific example is the sparse
logistic regression problem. This type of problem is often found in
data mining or machine learning, see e.g.
\cite{RyaSup:10,YuaCha:10}. In a training set $\left \{
a^j, b^j \right \}$, with $j \in [\bar{N}]$, the vectors $a^j \in
\rset^{n}$ represent $\bar{N}$ samples, and $b^j$ represent the
binary class labels with $b^j \in \left \{-1,+1 \right \}$. The
likelihood function for these $\bar{N}$ samples is:
\begin{equation*}
\sum_{j=1}^{\bar{N}} \mathbb{P} (b^j|a^j) ,
\end{equation*}
where $ \mathbb{P}(b|a)$ is the conditional probability and is expressed as:
\begin{equation*}
 \mathbb{P}(b|a)=\frac{1}{1+\exp(-b \langle a,x \rangle)},
\end{equation*}
with $x \in \rset^{n}$ being the weight vector. In some applications
(see e.g. \cite{YuaCha:10}), we require a bias term $c$ (also called
as an intercept) in the loss function; therefore, $\langle a,x
\rangle$ is replaced with $\langle a,x \rangle + c$. The equality
$\langle a, x\rangle=0$ defines a hyperplane in the feature space on
which $\mathbb{P}(b|a)=0.5$. Also, $\mathbb{P}(b|a)>0.5$ if $\langle
a,x \rangle>0$ and $\mathbb{P}(b|a)<0.5$ otherwise. Then, the sparse
logistic regression can be formulated as the following convex problem:
\begin{equation*}
 \min_{x \in \rset^{n}} f(x)+\lambda \left \|x\right \|_1,
\end{equation*}
where $\lambda>0$ is some constant and
$f(x)$ is the average logistic loss function:
\begin{equation*}
 f(x) =  - \frac{1}{\bar{N}} \sum_{i=1}^{\bar{N}}  \log \left (\mathbb{P}(b^j
 | a^j) \right) =  \frac{1}{\bar{N}} \sum_{j=1}^{\bar{N}} \log \left ( 1+\exp
 \left (-b^j \langle {a}^j, x \rangle  \right ) \right ).
\end{equation*}

\noindent Note that $\Psi(x)=\lambda \|x\|_1$, where $\|\cdot\|_1$ denotes the 1-norm, is the separable
non-smooth component which promotes the sparsity of the decision
variable $x$. If we associate to this problem a bipartite graph $G$
where the incidence matrix $E$ is defined such that $E_{ij} =1$
provided that ${a}^j_i \not=0$, then the vectors ${a}^j$ have a
certain sparsity according to this graph, i.e. they only have
nonzero components in ${a}^j_{\neigh_j}$. Therefore,  $f(x)$ can be
written as $f(x)=\sum_{j=1}^{\bar{N}} f_j(x_{\neigh_j})$, where each function $f_j$ is defined as:
\[   f_j(x_{\neigh_j}) =  \frac{1}{\bar{N}} \log \left ( 1+\exp
 \left (-b^j \langle {a}^j_{\neigh_j}, x_{\neigh_j} \rangle  \right ) \right ). \]
It can be easily proven that the objective function $f$ in this case
satisfies \eqref{lipschitz_comp} with $L_{\neigh_j}= \sum_{l \in
{\neigh}_j} \|{a}^j_l\|^2/4$ and \eqref{lipschitz_grad} with
$L_i=\sum_{j \in \bar{\neigh}_i} \|{a}^j_i\|^2/4$. Furthermore, we have that $f$ satisfies \eqref{desc_lemma} with
matrix $W=\text{diag}\left (\sum_{j \in
\bar{\neigh}_i} \|{a}^j_{\neigh_j}\|^2/4; \; i \in [n] \right )$.

\noindent \textit{Application II}: Another classic problem  which
implies functions $f_j$ with Lipschitz continuous gradient of type
\eqref{lipschitz_comp} is:
\begin{equation}
 \min_{x_i \in X_i \subseteq \rset^{n_i}} F(x) \quad
 \left (= \frac{1}{2} \| Ax-b\|^2 +\sum_{i=1}^N  \lambda_i \|x_i\|_1 \right),
 \label{sol_norm}
\end{equation}
where $A \in \rset^{\bar{N} \times n}$, the sets $X_i$  are convex,
$n= \sum_{i=1}^N n_i$ and $\lambda>0$. This problem is also known as
the constrained lasso problem \cite{JamPau:13} and is widely used
e.g. in signal processing, fused  or generalized lasso and monotone
curve estimation  \cite{JamPau:13,CheNg:12} or distributed model
predictive control \cite{NecClip:13a}. For example, in image
restoration, incorporating a priori information (such as box
constraints on $x$) can lead to substantial improvements in the
restoration and reconstruction process (see \cite{CheNg:12} for more
details).
Note that this problem is a special case of problem
\eqref{gen_form}, with $\Psi(x)= \sum_{i=1}^N [\lambda_i \|x_i\|_1 +
\bo{I}_{X_i}(x_i)] $ being block separable and with the functions $f_j$
defined as:
\begin{equation*}
 f_j(x_{\neigh_j})=\frac{1}{2} (a_{\neigh_j}^T x_{\neigh_j} - b_j)^2,
\end{equation*}
where $a_{\neigh_j}$ are the nonzero components of row  $j$ of $A$,
corresponding to $\neigh_j$. In this case, functions
$f_j$ satisfy \eqref{lipschitz_comp} with
$L_{\neigh_j} = \|a_{\neigh_j}\|^2$.  Given these constants, we find
that $f$ in this case satisfies \eqref{desc_lemma} with the matrix
$W=\text{diag}\left ( \sum_{j \in \bar{\neigh}_i} \|a_{\neigh_j}\|^2
I_{n_i}; \; i \in [N] \right)$. Also, note that functions of type
\eqref{sol_norm} satisfy Lipschitz continuity \eqref{lipschitz_grad}
with $L_i=\|A_i\|^2$, where $A_i \in \rset^{\bar{N} \times n_i}$
denotes block column $i$ of the matrix~$A$.

\noindent \textit{Application III}:  A third type of problem which
falls under the  same category is derived from the following  primal
formulation:
\begin{align}
 f^* =&\min_{u \in \rset^m} \; \sum_{j=1}^{\bar{N}} g_j(u_j), \label{prob_sum} \\
&\; \text{s.t:} \quad A u \leq b, \nonumber
\end{align}
where $A \in \rset^{n \times m},  u_j \in \rset^{m_j}$ and the
functions $g_j$ are strongly convex with convexity parameters
$\sigma_j$. This type of problem is often found in network control
\cite{NecSuy:09}, network optimization or  utility maximization
\cite{RamNed:09}. In all these applications matrix $A$ is very
sparse, i.e. both $(\omega, \bar \omega)$ are small. We formulate
the dual problem of \eqref{prob_sum} as:
\begin{equation*}
 \max_{x \in \rset^n} \left[ \min_{u \in \rset^m}
 \sum_{j=1}^{\bar{N}} g_j(u_j) + \langle x, Au - b \rangle \right] - \Psi(x),
\end{equation*}
where $x$ denotes the Lagrange multiplier and  $\Psi(x)=\bo{I}_{\rset^n_+}(x)$ is the indicator function for the nonnegative orthant $\rset^n_+$. By denoting by $g_j^*(z)$ the convex conjugate of the function  $g_j(u_j)$, the previous problem can be rewritten as:
\begin{align}
 f^*=& \max_{x \in \rset^n}  \left[   \sum_{j=1}^{\bar{N}}   \min_{u_j \in \rset^{m_j}}
  g_j(u_j) - \left \langle -A_j^T x , u_j \right \rangle \right] - \langle x, b \rangle - \Psi(x)
   \nonumber \\
=& \max_{x \in \rset^n} \sum_{j=1}^{\bar{N}}
  -g_j^*(-A_j^T x) -\langle x, b \rangle - \Psi(x), \label{conv_conjug}
\end{align}
where $A_j \in \rset^{n \times m_j}$ is the  $j$th block  column of
$A$. Note that, given the strong convexity of $g_j(u_j)$, then
the functions $g_j^*(z)$ have Lipschitz continuous  gradient  in $z$ of type
\eqref{lipschitz_comp} with constants $\frac{1}{\sigma_j}$
\cite{Nes:04}. Now, if the matrix $A$ has some sparsity induced by a
graph, i.e. the blocks $A_{ij}=0$ if the corresponding incidence matrix has $E_{ij}=0$, which in turn implies that the block columns $A_j$ are sparse
according to some index set $\neigh_j$, then the matrix-vector
products $A_j^T x$ depend only on $x_{\neigh_j}$, such that
$f_j(x_{\neigh_j}) = -g_j^*\left ( -A_{\neigh_j}^T x_{\neigh_j}
\right ) - \langle x_{\neigh_j}, \bar{b}_{\neigh_j} \rangle$, with
$\sum_j \langle x_{\neigh_j}, \bar{b}_{\neigh_j} \rangle  = \langle
x, b \rangle$. Then, $f_j$ has Lipschitz continuous gradient  of type
\eqref{lipschitz_comp} with $L_{\neigh_j} =
\frac{\|A_{\neigh_j}\|^2}{\sigma_j}$. For this problem we also have componentwise
Lipschitz continuous  gradient of type \eqref{lipschitz_grad} with
$L_i=\sum_{j \in \bar{\neigh}_i} \frac{\|A_{ij}\|^2}{\sigma_j}$. Furthermore, we
 find that  in this case $f$ satisfies \eqref{desc_lemma} with
$W = \text{diag}\left ( \sum_{j \in \bar{\neigh}_i}
\frac{\|A_{\neigh_j}\|^2}{\sigma_j}; \; i \in [n] \right)$. Note
that there are many applications  in distributed control or network
optimization where  $\omega$ or $ \bar \omega$ or both are small.
E.g., one particular application that appears in the area of network
optimization has the structure \eqref{prob_sum}, where the matrix
$A$  has column linked block angular form, i.e. the matrix $A$ has a
block  structure of the following form:
\begin{equation*}
 A=\left [ \begin{matrix}
   A_{11} & \\ A_{21} & A_{22} \\ A_{31} & & A_{33} \\ A_{n-1,n} &  & & A_{n-1,n-1} \\  A_{n1} &  & & & A_{nn}
 \end{matrix} \right ].
\end{equation*}
One of the standard distributed algorithms to solve network problems
is  based on a dual decomposition as explained above. In this case,
by denoting $g_j^*(z)$ the convex conjugate of the function
$g_j(u_j)$, the corresponding  problem can be rewritten as:
\begin{align}
 f^*= & \max_{x \in \rset^n}  \left[-g_1^*(-A_1^T x) - \sum_{j=2}^{\bar{N}}
  -g_j^*(-A_{jj}^T x_j) -\langle x, b \rangle - \Psi(x) \right].
\end{align}
If we consider the block columns of dimension $1$, then
$g_1^*(-A_1^T x)$ depends on all the coordinates of $x$, i.e.
$\omega=\bar{N}$. On the other hand, note that given the structure
of $A$, we have that $\bar{\omega}=2 \ll \omega = \bar N$. The
reader can easily find many other examples of objective functions
where $\bar{\omega} \ll \omega$.


\section{Parallel  random coordinate descent method}
\label{sec_rcdm} In this section we {\bl consider} a  parallel
version of the random coordinate descent method
\cite{Nes:12,NecPat:12,RicTak:12}, which we call \textbf{P-RCD}.
{\bl Analysis of coordinate descent methods based on updating in
parallel  more than one (block) component per iteration was given
first in \cite{NecNes:11,BraKyr:11,NecClip:13a} and then further
studied  e.g. in  \cite{RicTak:12a,PenYan:13}. In particular, such a
method and its convergence properties has  been analyzed in
\cite{NecClip:13a,RicTak:12a}, but under the coordinate-wise
Lipschitz assumption \eqref{lipschitz_grad}.} Before we {\bl
discuss} the method however, we first need to introduce some
concepts. For a function $F(x)$ as defined in \eqref{gen_form}, we
introduce the following mapping in norm $\|\cdot\|_W$:
\begin{equation}
 t_{[N]}(x,y)=f(x)+\langle \nabla f(x),y-x \rangle+\frac{1}{2}\|y-x\|^2_W+\Psi (y). \label{comp_grad}
\end{equation}
Note that the mapping $t_{[N]}(x,y)$ is a fully separable and strongly
convex in $y$ w.r.t. to the norm $\|\cdot\|_W$  with the constant 1.
We denote  by $T_{[N]}(x)$ the proximal step for function $F(x)$,
which is the optimal point of the mapping  $t_{[N]}(x,y)$, i.e.:
\begin{equation}\label{prox_optim}
T_{[N]}(x)=\arg \min_{y \in \rset^n} t_{[N]}(x,y).
\end{equation}
The proximal step $T_{[N]}(x)$ can also be defined in another way.
We define the proximal operator of function $\Psi$ as:
\begin{equation*}
 \prox(x)=\arg \min_{u \in \rset^n} \Psi(u)+\frac{1}{2} \|u-x\|_W^2.
\end{equation*}

\noindent We recall an important property of the proximal operator
\cite{RocWet:98}:
\begin{equation}
 \|\prox(x)-\prox(y)\|_W \leq \|x-y\|_W. \label{prox_property1}
\end{equation}

\noindent Based on this operator, note that we can write:
\begin{equation}
 T_{[N]}(x)=\prox(x-W^{-1} \nabla f(x)).  \label{T_prox}
\end{equation}
Given that $\Psi(x)$ is generally not differentiable, we  denote by
$\partial_\Psi(x)$ a vector belonging to the set of subgradients of
$\Psi(x)$. Evidently, in both definitions, the optimality conditions
of the resulting problem from which we obtain $ T_{[N]}(x)$ are the
same, i.e.:
\begin{equation}
0 \in \! \nabla f(x)\!+\!W(T_{[N]}(x)\!-\!x)\!+\!\partial_{\Psi}
(T_{[N]}(x)). \label{prox_cond}
\end{equation}

\noindent It will become evident  further on that the optimal
solution $T_{[N]}(x)$ will play a crucial role in the parallel
random coordinate descent method. We now establish some properties
which involve the function $F(x)$, the mapping $t_{[N]}(x,y)$ and
the proximal step  $T_{[N]}(x)$. {\bl Given that $t_{[N]}(x,y)$ is
strongly convex in $y$ and that $T_{[N]}(x)$ is an optimal point
when minimizing over $y$, we have the following inequality:}
\begin{align}
F(x)-t_{[N]}(x,T_{[N]}(x))&=t_{[N]}(x,x)-t_{[N]}(x,T_{[N]}(x)) \nonumber \\
&{\geq} \frac{1}{2} \|x-T_{[N]}(x)\|_W^2. \label{y_desc}
\end{align}
Further, since $f$ is convex and differentiable and by
definition of $t_{[N]}(x,y)$ we get:
\begin{align}
 t_{[N]}(x,T_{[N]}(x))&\leq \min_{y \in \rset^n} f(y)+\Psi (y)+\frac{1}{2}\|y-x\|_W^2 \nonumber \\
&=\min_{y \in \rset^n} F(y)+\frac{1}{2}\|y-x\|_W^2. \label{prox_h}
\end{align}

\noindent In the algorithm that we discuss, at a step $k$, the
components of the iterate $x^k$ which are to be updated are dictated
by a set of indices $\coord^k \subseteq [N]$ which is randomly
chosen.  Let us denote by $x_{\coord} \in \rset^n$  the vector whose
blocks $x_i$, with $i \in \coord \subseteq [N]$, are identical to
those of $x$, while the remaining blocks are zeroed out, i.e.
$x_{\coord} = \sum_{i \in \coord} U_i x_i$ or:
\begin{equation}
x_{\coord}=
 \begin{cases}
 x_i, & \; i \in \coord \\
 0, & \; \text{otherwise.}
 \end{cases}
\end{equation}
Also, for the separable function $\Psi(x)$, we denote the partial sum $\Psi_{\coord}(x) \!=\! \sum_{i \in \coord} \Psi_i(x_i)$ and the vector $\partial_{\coord} \Psi(x)=[\partial \Psi(x)]_{\coord} \in \rset^{n}$. A random variable $\coord$ is uniquely characterized by the probability density function:
\begin{equation*}
 P_{\hat{\coord}}=P(\coord=\hat{\coord}) ~~ \text{where} ~~ \hat{\coord} \subseteq [N].
\end{equation*}
For the random variable  $\coord$, we also define the probability with which a subcomponent $i \in [N]$ can be found in $\coord$ as:
\begin{equation*}
 p_i=\mathbb{P}(i \in \coord).
\end{equation*}
In our algorithm, we consider a uniform sampling of $\tau$ unique coordinates $i$, $1\leq \tau \leq N$ that make up $\coord$, i.e. $|\coord|=\tau$.
For a random variable $\coord$ with $|\coord|=\tau$, we observe that we have a total number of $\binom{N}{\tau}$ possible values that $\coord$ can take, and with the uniform sampling we have that $P_{\coord}=\frac{1}{\binom{N}{\tau}}$.
Given that $\coord$ is random, we can express $p_i$ as:
\begin{equation*}
 p_i=\sum_{\coord: \; i \in \coord } P_{\coord}.
\end{equation*}
For a single index $i$, note that we have a total number of $\binom{N-1}{\tau-1}$ possible sets that $\coord$ can take which will include $i$ and therefore the probability that this index is in $\coord$ is:
\begin{equation}
p_i=\frac{\binom{N-1}{\tau-1}}{\binom{N}{\tau}}=\frac{\tau}{N}. \label{prob_i}
\end{equation}

\begin{remark}
\label{remark_probabilities}
We can  also consider other ways in which $\coord$ can be chosen.
For example,  we can have partition sets $\coord^1, \dots,
\coord^{q}$ of $[N]$, i.e. $[N]=\bigcup_{i=1}^q \coord^i$, that are
randomly shuffled. We can  choose $\coord$ in a nearly
independent manner, i.e. $\coord$ is chosen with a sufficient
probability, or we can choose $\coord$ according to an irreducible
and aperiodic Markov chain, see e.g. \cite{TseYun:09,WanBer:13}.
However, if we  employ these strategies for choosing $\coord$,  the
proofs for the convergence rate of our algorithm follow similar~lines.
\end{remark}

\vspace{5pt}

\noindent Having defined the proximal step as $T_{[N]}(x^k)$ in
\eqref{prox_optim}, in the algorithm that follows we generate
randomly at step $k$  an index set $\coord^k$  of cardinality $1\leq
\tau \leq N$. We denote the vector $T_{\coord^k}(x^k)=[T_{[N]}
(x^k)]_{\coord^k}$ which will be used to update $x^{k+1}$, i.e. in
the sense that $[x^{k+1}]_{\coord^k}=T_{\coord^k}(x^k)$. Also, by
$\bar{\coord}^k$ we denote the complement set of $\coord^k$, i.e.
$\bar{\coord}^k=\{i \in [N]: i \notin \coord^k \}$. Thus, the
algorithm consists of the following steps:
\begin{center}
\framebox[\textwidth]{
\parbox{\textwidth}{
\begin{center}
\textbf{ Parallel random  coordinate descent method
(P-RCD)}
\end{center}
\begin{enumerate}
\item Consider an initial point $x^0 \in \rset^n$ and $1 \leq \tau \leq N$
\item For $k\geq 0$:
\begin{itemize}
\item[2.1] Generate  with uniform probability a random set of indices $\coord^k \subseteq [N]$, \\ with $|\coord^k|=\tau$
\item[2.2] Compute:
\end{itemize}
\begin{equation*}
x^{k+1}_{\coord^k}=T_{\coord^k}(x^k) \; \text{and} \; x^{k+1}_{\bar{\coord}^k}=x^k_{\bar{\coord}^k}.
\end{equation*}
\end{enumerate}
}}
\end{center}
Note that the iterate update of \textbf{(P-RCD}) method can also be expressed as:
\begin{align}
\begin{cases}
x^{k+1}=x^k+T_{\coord^k}(x^k)-x^k_{\coord^k} \\
 x^{k+1}=\arg \min_{y \in \rset^n} \langle \nabla_{\coord^k} f(x^k), y-x^k \rangle +\frac{1}{2} \|y-x^k\|_W^2+\Psi_{\coord^k} (y) \\
 x^{k+1}=\text{prox}_{\Psi_{\coord^k}}\left (x^k-W^{-1} \nabla_{\coord^k} f(x^k) \right ).
\end{cases} \label{iter_pcdm2}
\end{align}
Note that the right hand sides of the last two equalities contain the  same optimization problem whose optimality conditions are:
\begin{equation}
 \begin{cases}
W[x^k-x^{k+1}]_{\coord^k} \in \nabla_{\coord^k} f(x^k)+\partial \Psi_{\coord^k} (x^{k+1}) \\
[x^{k+1}]_{\bar{\coord}^k}=[x^k]_{\bar{\coord}^k}.
 \end{cases} \label{opt_prox_3}
\end{equation}

\noindent Clearly, the optimization  problem from which we compute
the iterate of \textbf{(P-RCD)} is fully separable. Then, it follows
that for updating component $i \in \coord^k$ of $x^{k+1}$ we need
the following data: $\Psi_i(x_i^k), W_{ii}$ and $\nabla_i f(x^k)$.
However, the $i$th diagonal entry $W_{ii} = \sum_{j \in
\bar{\neigh}_i} L_{\neigh_j}$  and $i$th block component of the
gradient $\nabla_i  f = \sum_{j \in \bar{\neigh}_i} \nabla_i f_j$
can be computed distributively according to the communication graph
$G$  imposed on the original  optimization problem. Therefore, if
algorithm \textbf{(P-RCD)} runs on a multi-core machine or as a
multi-thread process, it can be observed that component updates can
be done distributively and in parallel by each core/thread
(see Section \ref{sec_comparison} for details).  

\noindent We now establish that method \textbf{(P-RCD)} is a descent
method, i.e. $F(x^{k+1}) \leq F(x^k)$ for all  $k \geq 0$. \noindent
From the convexity of $\Psi(\cdot)$ and \eqref{desc_lemma} we obtain
the following:
\begin{align}
F(x^{k+1})&\leq F(x^k) \!+\! \langle \nabla f(x^k),x^{k+1} \!-\! x^k \rangle  \!+\! \langle \partial_{\Psi}(x^{k+1}),x^{k+1} \!-\! x^k\rangle \!+\! \frac{1}{2}\|x^{k+1} \!-\! x^k\|_W^2 \nonumber \\
&=F(x^k)+\langle \nabla_{\coord^k} f(x^k) + \partial_{ \Psi_{\coord^k}} (x^{k+1}), [x^{k+1}-x^k]_{\coord^k} \rangle +\frac{1}{2}\|x^{k+1}-x^k\|_W^2 \nonumber \\
&\hspace{-0.1cm} \overset{\eqref{opt_prox_3}}{=}F(x^k)+\langle W[x^k-x^{k+1}]_{\coord^k}, [x^{k+1}-x^k]_{\coord^k} \rangle +\frac{1}{2}\|x^{k+1}-x^k\|_W^2 \nonumber \\
&=F(x^k)-\frac{1}{2}\|x^{k+1}-x^k\|_W^2. \label{desc_total}
\end{align}
With  \textbf{(P-RCD)} being a descent method, we can now introduce
the following term:
\begin{equation}
 R_W(x^0)=\max_{x: \; F(x) \leq F(x^0)} \min_{x^* \in X^*} \|x-x^*\|_W. \label{bound_R}
\end{equation}
and assume it to be bounded.  We also define the random variable
comprising the whole history of previous events as:
\begin{equation*}
\eta^k = \{ \coord^0, \dots, \coord^k \}.
\end{equation*}


\section{Sublinear convergence for smooth convex minimization}
\label{sec_smooth} In this section we establish the sublinear
convergence rate  of method \textbf{(P-RCD)} for problems of type
\eqref{gen_form} with the objective function satisfying Assumption
\ref{lip_fi}.  {\bl Our analysis in this section combines the tools
developed above with the convergence analysis  in \cite{LuXia:13}
for random one block coordinate descent methods}.  The next lemma
provides some property for the uniform sampling with
$|\coord|=\tau$:
\begin{lemma}  \cite[Lemma 3]{RicTak:12a}
\label{sep_probab} Let there be some constants $\theta_i$ with
$i=1,\dots,N$, and a sampling $\coord$ chosen as described above and
define the sum $\sum_{i \in \coord} \theta_i$. Then,
the expected value of the sum  satisfies:
\begin{equation}
 \average \left [ \sum_{i \in \coord} \theta_i \right ] = \sum_{i=1}^N p_i \theta_i.
 \label{probab_sep}
\end{equation}
\end{lemma}

\noindent For any vector $d \in \rset^n$ we consider its counterpart
$d_{\coord}$ for a sampling $\coord$ taken as
described above.  Given the previous lemma and by taking into account the
separability of the inner product and  of the squared
norm $\|\cdot\|_W^2$ it follows immediately  that:
\begin{align}
& \mathbb{E}\left [\langle x,d_{\coord} \rangle \right ] = \frac{\tau}{N}
 \langle x,d \rangle \label{sep_prop2}\\
& \mathbb{E}\left [\|d_{\coord}\|_W^2 \right ] = \frac{\tau}{N}
\|d\|_W^2.  \label{sep_prop3}
\end{align}

\noindent Based on relations \eqref{sep_prop2}--\eqref{sep_prop3}, the
separability  of the function $\Psi(x)$, and the properties of the expectation operator,
 the following inequalities can be  derived
(see e.g. \cite{RicTak:12a}):
\begin{align}
& \mathbb{E} \left [\Psi(x+d_{\coord}) \right ] =
\frac{\tau}{N} \Psi(x+d)+\!\left (1-\frac{\tau}{N} \right)\Psi(x)
\label{sep_prop4} \\
& \average \left [  F(x + d_{\coord}) \right ] \leq \left ( 1-\frac{\tau}{N} \right ) F(x) +\frac{\tau}{N} t_{[N]}(x,d). \label{sep_essential}
\end{align}

\noindent We can now formulate an important relation between the
gradient mapping  in a point $x$ and a point $y$.
By the definition of $t_{[N]}(x,y)$ and the convexity of $f$ and $\Psi$ we have:
\begin{align*}
& t_{[N]}(x,T_{[N]}(x)) = f(x) + \langle \nabla f(x),T_{[N]}(x) - x\rangle + \frac{1}{2} \|T_{[N]}(x) - x\|_W^2 + \Psi(T_{[N]}(x)) \\
& \quad \leq f(y) + \langle \nabla f(x),x - y \rangle + \langle \nabla f(x),T_{[N]}(x)-x\rangle + \frac{1}{2} \|T_{[N]}(x) - x\|_W^2 + \Psi(T_{[N]}(x)) \\
& \quad \leq f(y) +  \langle \nabla f(x),x - y \rangle + \langle \nabla f(x),T_{[N]}(x) - x\rangle+ \frac{1}{2} \|T_{[N]}(x) - x\|_W^2 \\
& \qquad + \Psi(y) + \left \langle \partial \Psi(T_{[N]}(x)),T_{[N]}(x)-y \right \rangle.
\end{align*}
Furthermore, from the optimality conditions \eqref{prox_cond} we obtain:
\begin{align}
& t_{[N]}(x,T_{[N]}(x)) \!\leq\!  f(y) \!+\!\! \langle \nabla f(x),x \!\!-\!\! y \rangle \!\!+\!\! \langle \nabla f(x),T_{[N]}(x) \!\!-\!\! x\rangle \nonumber \\
&\qquad \!\!+\!\! \frac{1}{2} \|T_{[N]}(x) \!-\! x\|_W^2 + \Psi(y) +\left \langle -\nabla f(x) - W \left (T_{[N]}(x) - x \right), T_{[N]}(x) - y \right \rangle \nonumber \\
&\quad = F(y) + \left \langle \nabla f(x),x - y \right \rangle + \left \langle \nabla f(x), T_{[N]}(x) - x \right \rangle + \frac{1}{2} \|T_{[N]}(x) - x\|_W^2 \nonumber \\
& \qquad +\!\! \left \langle \!-\! \nabla f(x) \!\!-\!\! W \left(T_{[N]}(x) \!\!-\!\! x \right )\!,\! T_{[N]}(x) \!\!-\!\! x \right \rangle  \!\!+\!\! \left \langle \!-\! \nabla f(x) \!\!-\!\! W\left (T_{[N]}(x) \!\!-\!\! x \right )\!,\! x \!\!-\!\! y \right \rangle \nonumber \\
 &\quad = F(y)-\left \langle W \left( T_{[N]}(x)-x \right),x-y \right \rangle -\frac{1}{2} \|T_{[N]}(x)\!-\!x\|_W^2. \label{new_prop}
\end{align}

\noindent This property will prove useful in the following  theorem,
which  provides the sublinear convergence rate for method
\textbf{(P-RCD)}.

\begin{theorem}\label{theorem_sublinear}
If Assumption \ref{lip_fi} holds  and  considering that $R_W(x^0)$
defined in \eqref{bound_R} is bounded, then for the sequence  $x^k$
generated by algorithm \textbf{(P-RCD)} we have:
\begin{equation}
 \average [F(x^{k})] - F^* \leq \frac{ N \left( 1/2(R_W(x^0))^2 +  F(x^0)-F^* \right)}{\tau k + N}.
 \label{pcdm_conv_1}
\end{equation}
\end{theorem}

\begin{proof}  Our proof generalizes    the proof of Theorem 1 in \cite{LuXia:13}
from one (block) component update per iterate to  the case of $\tau$
(block) component updates, {\bl  based on uniform sampling and on Assumption \ref{lip_fi}, and consequently on a different  descent
lemma}.  Thus, by taking expectation in \eqref{desc_total} w.r.t. $\coord^k$ conditioned on $\eta^{k-1}$ we get:
\begin{equation}
\average [F(x^{k+1})] \leq  F(x^k)-\frac{1}{2}  \average  \left [ \|x^{k+1}-x^k\|_W^2 \right] \leq F(x^k).  \label{desc_average}
\end{equation}
Now, if we take $x=x^k$ and
$y_{\coord^k}=T_{\coord^k}(x^k)-x^k_{\coord^k}$ in
\eqref{sep_essential}  we get:
\begin{equation}
\average \!\left [ F(x^{k+1}) \right ]\leq \left (1-\frac{\tau}{N} \right)F(x^k)+\frac{\tau}{N} t_{[N]}\left (x^k,T_{[N]}(x^k) \right ). \label{ineq_essential}
\end{equation}
From this and \eqref{new_prop} we obtain:
\begin{align}
\frac{\tau}{N} F(y)+\frac{N-\tau}{N} F(x^k) \geq & \average \left [ F(x^{k+1}) \right ] +
\frac{\tau}{N} \left \langle W \left( T_{[N]}(x^k)-x^k \right),x^k-y  \right\rangle \label{new_prop2} \\
 & +\frac{\tau}{2N} \|T_{[N]}(x^k)\!-\!x^k\|_W^2. \nonumber
\end{align}

\noindent Denote $r^k=\|x^k-x^*\|_W$. From the definition of $x^{k+1}$ we have that:
\begin{equation*}
(r^{k+1})^2=(r^k)^2+ \sum_{i \in \coord^k} \left[ 2W_{ii} \langle T_{i}(x^k)-x_i^k , x_i^k-x_i^*
\rangle +W_{ii} \|T_{i}(x^k)-x^k_i \|^2 \right].
\end{equation*}
If we divide both sides of the above inequality by $2$ and take expectation, we obtain:
\begin{equation*}
\average \left [\frac{1}{2}(r^{k+1})^2 \right ]=\frac{(r^k)^2}{2}+
\frac{\tau}{N} \left \langle W \left( T_{[N]}(x^k)-x^k \right ), x^k - x^* \right \rangle +\frac{\tau}{2N}  \|T_{[N]}(x^k)-x^k \|^2_W.
\end{equation*}
Through this inequality and \eqref{new_prop2} we arrive at:
\begin{equation*}
\average \left [\frac{1}{2}(r^{k+1})^2 \right ]\leq \frac{(r^k)^2}{2}+\frac{\tau}{N} F^*+\frac{N-\tau}{N} F(x^k)-\average \left [F(x^{k+1}) \right].
\end{equation*}
After some rearranging of terms we obtain the following inequality:
\begin{equation*}
\average \left [\frac{1}{2}(r^{k+1})^2+F(x^{k+1})-F^* \right ]\leq \left ( \frac{(r^k)^2}{2}+F(x^k)-F^* \right )-\frac{\tau}{N}(F(x^k)-F^*).
\end{equation*}
By applying this inequality repeatedly, taking expectation over $\eta^{k-1}$ and from the fact that $\average [F(x^k)]$ is decreasing from \eqref{desc_average}, we obtain the following:
\begin{align*}
 \average \left [F(x^{k+1}) \right ]-F^* & \leq \average
 \left [\frac{1}{2}(r^{k+1})^2+F(x^{k+1})-F^* \right ] \\
 & \leq \frac{(r^0)^2}{2}+F(x^0)-F^*-\frac{\tau}{N} \sum_{j=0}^k \left ( \average \left [F(x^j) \right ]-F^* \right) \\
& \leq \frac{(r^0)^2}{2}+F(x^0)-F^*-\frac{\tau(k+1)}{N} \left (\average \left [ F(x^{k+1}) \right ]-F^* \right ).
\end{align*}
By rearranging some items and since $(r^0)^2\leq (R_W(x^0))^2$,
we arrive at \eqref{pcdm_conv_1}.
\end{proof}

\noindent We notice that given the choice of  $\tau$ we get
different results (see Section \ref{sec_comparison} for a detailed
analysis).  We also notice that the convergence rate depends on the choice of $ \tau =
|\coord|$, so that if the algorithm is implemented on a multi-core
machine or cluster, then $\tau$ reflects the available number of
cores.

\noindent Now, given a suboptimality level $\epsilon$ and a confidence
level $0 < \rho < 1$, we can establish a total number of iterations
$k^{\epsilon}_{\rho}$ which will ensure an $\epsilon$-suboptimal
solution with probability at least $1-\rho$.
\begin{corollary}
Under Assumption \ref{lip_fi}  and  with $R_W(x^0)$ defined in
\eqref{bound_R} bounded, consider a suboptimality level $\epsilon$
and a probability level $\rho$. Then,  for the iterates generated by
algorithm \textbf{(P-RCD)} and a  $k^{\epsilon}_{\rho}$ that
satisfies
\begin{equation}
 k^{\epsilon}_{\rho} \geq \frac{  c  }{ \epsilon}\left ( 1+\log \left ( \frac{N}{\tau} \frac{\left (R_W(x^0) \right)^2+2 \left (F(x^0)-F^* \right )}{4c \rho} \right ) \right )+2-N, \label{rez_prob}
\end{equation}
with $c = \frac{2N}{\tau} \max \left \{(R_W(x^0))^2,F(x^0)-F^*
\right \}$, we get the following result in probability:
\begin{equation}
 \mathbb{P} \left (F\left (x^{k^{\epsilon}_{\rho}} \right )-F^* \leq \epsilon \right ) \geq 1-\rho. \label{result_probab}
\end{equation}
\end{corollary}

\begin{proof}
By denoting $\delta^k=F(x^k)-F^*$, from \eqref{desc_average} we have that
$\delta^{k+1} \leq \delta^k$ and then it can be proven that
 $\delta^k$ satisfies (see e.g. \cite{RicTak:12a}):
\[ \average [\delta^{k+1}] \leq \left (1-\frac{\delta^k}{c} \right)\delta^k, \]
with $c$ defined above. From this and \eqref{pcdm_conv_1}, we
can choose $\epsilon \leq \delta^0$ and following the proof of
\cite[Theorem 2]{LuXia:13}, then for $ k^{\epsilon}_{\rho} $ defined in
 \eqref{rez_prob} we obtain \eqref{result_probab}.
\end{proof}


\section{Linear convergence for  error bound convex minimization}
\label{sec_gebf}
In this section we prove that, for certain minimization problems,
the sublinear convergence rate of \textbf{(P-RCD)} from the previous
section can be improved to a linear convergence rate. In particular,
we prove that under additional  assumptions on the objective
function, which are often satisfied in practical applications (e.g.
the dual of a linearly constrained smooth convex problem or control
problem), we have a \textit{generalized error bound property} for
our problem. In these settings we analyze the convergence behavior
of algorithm   \textbf{(P-RCD)} for which we are able to provide
\textit{global} linear convergence rate, as opposed to the results
in \cite{LuoTse:93,TseYun:09} where only \textit{local} linear
convergence was derived for deterministic descent methods or the
results in \cite{WanLin:13} where global linear convergence is
proved for a gradient method but applied only to problems where
$\Psi$ is the set indicator function of a  polyhedron.


\subsection{Linear convergence in the strongly convex case}
We assume that the function $f$ is additionally strongly convex in
the norm $\|\cdot\|_W$ with a constant $\sigma_W$, i.e.:
\begin{equation}
 f(y) \geq f(x)+\langle \nabla f(x),y-x \rangle+\frac{\sigma_W}{2} \|y-x\|_W^2. \label{str_convex}
\end{equation}
and since the function $f$ has Lipschitz continuous gradient  in the
the norm $\|\cdot\|_W$ with constant $1$, then it automatically
holds that $\sigma_W \leq 1$. Since $F$ is strongly convex function,
it follows that optimization problem \eqref{gen_form} has a unique
optimal point $x^*$.  The following theorem provides the convergence
rate of algorithm (\textbf{P-RCD}) when $f$ satisfies
\eqref{str_convex} and its proof follows the lines of the proof in
\cite{Nes:13}.

\begin{theorem}\label{lin_converg_strong}
For optimization problem \eqref{gen_form} with the objective
function satisfying  Assumption \ref{lip_fi} and the strong
convexity property \eqref{str_convex}, the algorithm
\textbf{(P-RCD)} has  global linear convergence rate for the
expected values of the objective function:
\begin{equation}
 \average \left [ F(x^{k})-F^* \right ] \leq (\theta_{sc})^k  (F(x^0)-F^*)   \qquad   \forall k \geq 0, \label{converg_linear_strong}
\end{equation}
where $\theta_{sc}  =\tau \sigma_W/N < 1 $.
\end{theorem}

\begin{proof}
 If we subtract $F^*$ from both sides of inequality \eqref{ineq_essential} we have that:
\begin{equation}
\average \!\left [ F(x^{k+1}) \right ] -F^*\leq \left
(1-\frac{\tau}{N} \right) \left (F(x^k)-F^* \right )+\frac{\tau}{N}
\left ( t_{[N]}(x^k,T_{[N]}(x^k))-F^* \right ). \label{ineq_sc_1}
\end{equation}
From the definition of $t_{[N]}(x,y)$ in \eqref{comp_grad}, of
$T_{[N]}(x)$ in \eqref{prox_optim} and from the strong convexity
\eqref{str_convex} we have that:
\begin{align*}
 t_{[N]}(x^k,T_{[N]}(x^k)) &=\min_{y \in \rset^n} f(x^k)+\langle \nabla f(x^k),y-x \rangle+\frac{1}{2}\|y-x^k\|^2_W+\Psi (y) \\
& \overset{\eqref{str_convex}}{\leq } \min_{y \in \rset^n} F(y)
+\frac{\left (1-\sigma_W \right )}{2} \left \|y-x^k \right \|_W^2.
\end{align*}
We now consider  $y$  of the form $\alpha x^*+(1-\alpha)x^k$ for
$\alpha \in [0,1]$ and since the functions above are convex, then we
have that:
\begin{align}
 t_{[N]}(x^k,T_{[N]}(x^k)) \leq \min _{\alpha \in [0,1]} F(\alpha x^* +(1-\alpha)x^k) +\frac{\alpha^2 (1-\sigma_W)}{2} \left \|x^k-x^* \right \|_W^2. \label{ineq_interm_sc}
\end{align}
Now, since $f$ is strongly convex, it satisfies \eqref{str_convex}
and also the following inequality for any $x, y$ and $\alpha \in
[0,1]$ (see e.g. \cite{Nes:04}):
\begin{equation*}
 f(\alpha x +(1-\alpha) y) \leq \alpha f(x) +(1-\alpha) f(y) -\frac{\alpha(1-\alpha) \sigma _W}{2} \left \|x-y \right \|_W^2.
\end{equation*}
In this inequality, if we take $x=x^*$, $y=x^k$, and considering
that $\Psi$ is also convex, we get that:
\begin{equation*}
 F(\alpha x^*+(1-\alpha) x^k) \leq \alpha F(x^*) +(1-\alpha) F(x^k)- \frac{\alpha(1-\alpha)\sigma_W}{2} \left \|x^k - x^* \right \|_W^2.
\end{equation*}
By replacing  this inequality in \eqref{ineq_interm_sc} we get the
following:
\begin{align*}
  t_{[N]}(x^k,T_{[N]}(x^k)) \leq \min_{\alpha \in [0,1]} \alpha F(x^*) +(1-\alpha) F(x^k) + \frac{\alpha(\alpha-\sigma_W)}{2} \left \|x^k - x^* \right \|_W^2.
\end{align*}
Now, we can choose a feasible  $\alpha_0=\sigma_W<1$, and therefore
we get:
\begin{align*}
  t_{[N]}(x^k,T_{[N]}(x^k)) \leq  \sigma_W F(x^*) +(1-\sigma_W) F(x^k).
\end{align*}
From this and \eqref{ineq_sc_1} we obtain:
\begin{align}
\average \!\left [ F(x^{k+1}) \right ] -F^* & \leq \left (1-\frac{\tau}{N}+ \frac{\tau (1-\sigma_W) }{N} \right)  \left (F(x^k)-F^* \right ) \label{linear_sc} \\
& =  \left (1- \frac{\tau \sigma_W}{N} \right)  \left (F(x^k)-F^*
\right ) = \theta_{sc} \left (F(x^k)-F^* \right ), \nonumber
\end{align}
where we note that $\theta_{sc} =\tau \sigma_W/N < 1$. Now if we
denote $\delta^k=F(x^{k}) -F(\bar{x}^*) $, then by taking
expectation over $\eta^{k-1}$  in \eqref{linear_sc} we arrive at:
\begin{equation*}
 \average [\delta^{k}] \leq \theta_{sc} \average [\delta^{k-1}] \leq
 \dots \leq (\theta_{sc})^k \average [\delta^0],
\end{equation*}
and linear convergence is proved.
\end{proof}


\subsection{Linear convergence in the generalized
error bound property  case}

\noindent  We introduce the proximal gradient mapping  of function
$F(x)$:
\begin{equation}
 \nabla^+ F(x)=x-\prox\left (x-W^{-1}\nabla f(x) \right ). \label{prox_mapping}
\end{equation}
Clearly, a point $x^*$ is  an optimal solution of problem
\eqref{gen_form} if and only if $\nabla^+ F(x^*)=0$. In the next
definition we introduce the  {\it Generalized Error Bounded Property
(GEBP)}:
\begin{definition}
\label{error_bound} Problem  \eqref{gen_form} has the generalized
error bound property (GEBP) w.r.t.  the norm $\|\cdot\|_W$ if there
exist two nonnegative constants $\kappa_1$ and $\kappa_2$ such that
the composite objective function $F$ satisfies the  relation (we use
$\bar{x}=\Pi_{X^*}^W (x)$):
\begin{equation}
 \|x-\bar{x}\|_W \leq \left ( \kappa_1+\kappa_2 \|x-\bar{x}\|_W^2 \right ) \|\nabla^+ F(x)\|_W \quad \forall x \in \rset^n. \label{EEBF}
\end{equation}
\end{definition}

\noindent Note that the class of functions introduced  in
\eqref{EEBF} includes other known categories of  functions. For
example, functions $F$ composed of a  strongly convex function $f$
with a convex constant $\sigma_{W}$  w.r.t. the norm $\|\cdot\|_W$
and a general convex function $\Psi$ satisfy  our definition
\eqref{EEBF} with $\kappa_1= \frac{2}{\sigma_W}$  and $\kappa_2=0$,
see Section \ref{sec_gebfproperty} for more details.

\noindent  Next,  we prove that on  optimization problems having the
(GEBP) property \eqref{EEBF}  our algorithm (\textbf{P-RCD}) still
has global linear convergence. Our analysis  will employ ideas from
the convergence proof of deterministic descent methods in
\cite{TseYun:09}. However, the random nature of our method and the
the nonsmooth property of the objective function  requires a new
approach. For example,  the typical proof  for linear convergence of
gradient descent type methods for solving convex problems with an
error bound like property is based on deriving an inequality  of the
form $F(x^{k+1}) - F^* \leq c \| x^{k+1} - x^k \|$ (see e.g.
\cite{LuoTse:93, TseYun:09, WanLin:13}). Under our  settings, we
cannot derive this type of inequality but  instead we obtain a
weaker inequality, where we replace $\| x^{k+1} - x^k \|$ with
another term and which still allows us to prove linear convergence.
We start with the following lemma which shows an important property
of algorithm \textbf{(P-RCD)} when it is applied to problems having
generalized error bound objective function:
\begin{lemma}
If our problem \eqref{gen_form} satisfies (GEBP) given in
\eqref{EEBF}, then a point $x^k$ generated by algorithm
\textbf{(P-RCD)} and its projection onto $X^*$, denoted $\bar{x}^k$,
satisfy the following:
\begin{equation}
\|x^k - \bar{x}^k\|_W^2  \leq  \left( \kappa_1 + \kappa_2 \|x^k - \bar{x}^k\|_W^2 \right)^2
\frac{ N}{\tau} \average \left [\|x^{k+1} - x^k\|_W^2 \right ].
\label{x_diff_bounded}
\end{equation}
\end{lemma}
\begin{proof}
For the iteration defined by algorithm  \textbf{(P-RCD)} we have:
\begin{align}
 \average \left [ \|x^{k+1}-x^k\|_W^2 \right ]&=\average \left [ \|x^k+T_{\coord^k}(x^k)-x^k_{\coord^k}-x^k\|_W^2 \right ] \nonumber \\
&=\average \left [\|x^k_{\coord^k}-T_{\coord^k}(x^k)\|_W^2 \right ] \overset{\eqref{sep_prop3}}{=}\frac{\tau}{N} \|x^k-T_{[N]}(x^k) \|_W^2 \nonumber  \\
&=\frac{\tau}{N} \|x^k-\prox(x^k-W^{-1}\nabla f(x^k))\|_W^2 \nonumber \\
&=\frac{\tau}{N} \|\nabla^+ F(x^k)\|_W^2. \nonumber
\end{align}
Through this equality and \eqref{EEBF} we have that:
\begin{align}
 \|x^k-\bar{x}^k\|_W^2 & \leq (\kappa_1+\kappa_2 \|x^k-\bar{x}^k\|^2_W)^2 \|\nabla F^+(x^k)\|_W^2 \nonumber\\
& \leq (\kappa_1+\kappa_2 \|x^k-\bar{x}^k\|_W^2)^2 \frac{ N}{\tau} \average \left [\|x^{k+1}-x^k\|_W^2 \right ] ,
\end{align}
and the proof is complete.
\end{proof}

\begin{remark}
Note that if the iterates of an algorithm satisfy the following relation:
\begin{equation*}
 \|x^k - x^* \|\leq \|x^0 - x^*\| \quad \forall k \geq 1,
\end{equation*}
see e.g.  the case of the full gradient method \cite{Nes:04}, then we have:
\begin{equation}
 \|x^k-\bar{x}^k\|_W^2 \leq \frac{\bar{\kappa}(x^0)N}{\tau} \average \left [ \|x^{k+1}-x^k\|^2_W \right ] \quad \forall k \geq 0, \label{error_boundedness}
\end{equation}
where $\bar{\kappa}(x^0)=\left (\kappa_1+\kappa_2 \|x^0-x^*\|^2_W \right)^2$.

\noindent  If the iterates of an  algorithm satisfy \eqref{bound_R}
with $R_W(x^0)$ bounded, see e.g.  the case of our algorithm
\textbf{(P-RCD)} which is a descent method, as proven in
\eqref{desc_total}, then \eqref{error_boundedness} is satisfied with
$\bar{\kappa}(x^0)=(\kappa_1+\kappa_2 (R_W(x^0))^2)^2$.
\end{remark}

\noindent Let us now note that  given the  separability of function  $\Psi:\rset^n \rightarrow \rset^{}$, then for any vector $d \in \rset^n$ if we consider their counterparts $\Psi_{\coord}$ and $d_{\coord}$ for a sampling $\coord$ taken as
described above the expected value  $ \average \left [
\Psi_{\coord}  (d_{\coord}) \right]$ satisfies:
\begin{align}
\label{psi_average}
 \average [\Psi_{\coord}(d_{\coord})]& \!=\!\!\! \sum_{\coord \subseteq [N]} \!\! \left ( \! \sum_{i \in \coord}
  \Psi_i(d_i) \! \right)\! P_{\coord} \!\!=\!\!\! \sum_{\coord \subseteq [N]} \!\! \left (\! \sum_{i=1}^{N}
  \Psi_i(d_i) \bo{1}_{\coord}(i) \! \right) \!P_{\coord}  \\
 & =\sum_{i=1}^N \Psi_i(d_i) \!\!\!
  \sum_{\coord  \subseteq [N]:i \in \coord} \!\!\! P_{\coord} = \sum_{i=1}^N p_i \Psi_i(d_i)
\overset{\eqref{prob_i}}{=}\frac{\tau}{N} \sum_{i=1}^N
\Psi_i(d_i)=\frac{\tau}{N} \Psi(d). \nonumber
\end{align}
Furthermore, considering that $\bar{x}^k \in X^*$, then
from \eqref{x_diff_bounded} we obtain:
\begin{equation}
\left \|x^k-\bar{x}^k \right \|_W  \leq c_\kappa(\tau) \sqrt{
\average \left [\|x^{k+1}-x^k\|^2 \right ]}, \label{x_diff_bounded2}
\end{equation}
where $c_{\kappa}(\tau)=\left ( \kappa_1+\kappa_2 (R_W(x^0))^2
\right ) \sqrt{\frac{ N}{\tau}}$. We now need to express $\average
[\Psi(x^{k+1})]$ explicitly, where $x^{k+1}$ is generated by \textbf{(P-RCD)}. Note that
$x^{k+1}_{\bar{\coord}^k}=x^k_{\bar{\coord}^k}$. As a result:
\begin{align}
 \average [\Psi(x^{k+1})]&=\average \left [ \sum_{i \in \coord^k} \Psi_i \left ([T_{\coord^k}(x^k)]_i \right )+ \sum_{i \in \bar{\coord}^k} \Psi_i \left ([x^k]_i \right )  \right ] \nonumber \\
& \overset{\eqref{psi_average}}{=}\frac{\tau}{N} \Psi \left (T_{[N]}(x^k) \right )+\frac{N-\tau}{N} \Psi (x^k).  \label{psi_separated}
\end{align}

\noindent The following lemma  establishes an important upper bound
for $\average \! \left[ F(x^{k+1}) \!-\! F(x^k) \right ]$.
\begin{lemma}
If function $F$ satisfies  Assumption \ref{lip_fi} and the (GEBF)
property defined in \eqref{EEBF} holds, then  the iterate $x^k$
generated by \textbf{(P-RCD)} method has the following property:
\begin{equation}
\average \left [ F(x^{k+1})-F(x^k) \right ] \leq \average [
\Lambda^k] \qquad \forall k \geq 0,
 \label{delta_bound}
\end{equation}
where
$$\Lambda^k =\langle \nabla f(x^k),x^{k+1}-x^k\rangle +\frac{1}{2} \|x^{k+1}-x^k\|_W^2 +\Psi(x^{k+1})-\Psi(x^k). $$
Furthermore, we have that:
\begin{equation}
 \frac{1}{2} \|x^{k+1}-x^k\|_W^2 \leq - \Lambda^k \quad \forall k \geq 0.  \label{delta_norm_bound}
\end{equation}
\end{lemma}

\begin{proof}
Taking $x=x^k$ and
$y=x^{k+1}-x^k$ in  \eqref{desc_lemma} we get:
\begin{equation*}
 f(x^{k+1})\leq f(x^k)+\langle \nabla f(x^k), x^{k+1}-x^k \rangle+\frac{1}{2}\|x^{k+1}-x^k\|_W^2.
\end{equation*}
By adding $\Psi(x^{k+1})$ and substracting $\Psi(x^k)$ in both sides
of this inequality and by taking expectation in both sides we obtain
\eqref{delta_bound}. Recall the iterate update \eqref{iter_pcdm2}:
\begin{equation*}
 x^{k+1} = \arg \min_{y \in \rset^n} \; \langle \nabla_{\coord^k} f(x^k) , y-x^k \rangle+\frac{1}{2}\|y-x^k\|_W^2 +\Psi_{\coord^k} (y).
\end{equation*}
Given that $x^{k+1}$ is optimal for the problem above and if we take a vector $y=\alpha x^{k+1}+(1-\alpha) x^k$, with $\alpha \in [0,1]$, we have that:
\begin{align*}
 & \langle \nabla_{\coord^k} f(x^k),x^{k+1}-x^k \rangle +\frac{1}{2}\|x^{k+1}-x^k\|_W^2 +\Psi_{\coord^k} (x^{k+1}) \\
& \leq  \alpha \langle \nabla_{\coord^k} f(x^k), x^{k+1}-x^k \rangle +\frac{\alpha^2}{2} \|x^{k+1}-x^k\|_W^2  +\Psi_{\coord^k} (\alpha x^{k+1}+(1-\alpha) x^k).
\end{align*}
Further, if we rearrange the terms and through the convexity of $\Psi_{\coord^k}$ we obtain:
\begin{equation*}
  (1-\alpha) \! \left [ \langle  \nabla_{\coord^k} f(x^k),x^{k+1} \!-\! x^k \rangle + \frac{1+\alpha}{2} \|x^{k+1} \!-\! x^k\|_W^2 + \Psi_{\coord^k}(x^{k+1}) \!-\! \Psi_{\coord^k}(x^k) \right ] \leq 0.
\end{equation*}
If we divide this inequality by $(1-\alpha)$ and let $\alpha \uparrow 1$ we have that:
\begin{align*}
\langle \nabla_{\coord^k} f(x^k),x^{k+1}-x^k \rangle +(\Psi_{\coord^k}(x^{k+1})- \Psi_{\coord^k}(x^k))  \leq -\|x^{k+1}-x^k\|_W^2.
\end{align*}
By adding $\frac{1}{2} \|x^{k+1}-x^k\|_W^2$ in both sides of this inequality and observing that:
\begin{align*}
& \langle \nabla_{\coord^k} f(x^k),x^{k+1}-x^k \rangle=\langle \nabla f(x^k),x^{k+1}-x^k \rangle \quad \text{and} \\
& \Psi_{\coord^k}(x^{k+1})- \Psi_{\coord^k}(x^k)=\Psi(x^{k+1})-\Psi(x^k),
\end{align*}
we obtain \eqref{delta_norm_bound}.
\end{proof}

\noindent Additionally, note that by applying expectation in
$\coord^k$ to $\Lambda^k$ we get:
\begin{align}
 \average [ \Lambda^k ] & \overset{\eqref{sep_prop2},\eqref{sep_prop3}}{=}\frac{\tau}{N} \langle \nabla f(x^k), T_{[N]}(x^k)-x^k \rangle +\frac{1}{2} \average \left [  \|x^{k+1}-x^k\|_W^2 \right ] \nonumber \\
& \quad \qquad + \average \left [ \Psi(x^{k+1}) \right ] -\Psi(x^k) \nonumber \\
& \quad \overset{\eqref{psi_separated}}{=}\frac{\tau}{N} \langle \nabla f(x^k), T_{[N]}(x^k)-x^k \rangle +\frac{1}{2} \average \left [  \|x^{k+1}-x^k\|_W^2 \right ]  \label{delta_average} \\
&\quad  \qquad + \frac{\tau}{N} \left (\Psi(T_{[N]}(x)) -\Psi(x^k) \right ). \nonumber
\end{align}

\noindent The  following theorem, which is the main result of this
section, proves the linear convergence rate for the algorithm
\textbf{(P-RCD)} on optimization problems  having the generalized
error bound property \eqref{EEBF}.
\begin{theorem}\label{lin_converg}
On optimization problems \eqref{gen_form} with the objective
function satisfying  Assumption \ref{lip_fi} and the generalized
error bound property  \eqref{EEBF}, the algorithm \textbf{(P-RCD)}
has the following global linear convergence rate for the expected
values of the objective function:
\begin{equation}
 \average \left [ F(x^{k})-F^* \right ] \leq \theta^k  (F(x^0)-F^*)   \qquad   \forall k \geq 0, \label{converg_linear}
\end{equation}
where $\theta<1$ is a constant depending on $N, \tau, \kappa_1,\kappa_2$ and $R_W(x^0)$.
\end{theorem}
\begin{proof}
We first need to establish an upper bound for $\average [F(x^{k+1})]-F(\bar{x}^k)$.
By the definition of $F$ and its convexity we have that:
\begin{align*}
&F(x^{k+1})-F(\bar{x}^k) \\
&\!=f(x^{k+1})-f(\bar{x}^k)+\Psi(x^{k+1})-\Psi(\bar{x}^k) \\
&\! \leq \langle \nabla f(x^{k+1}),x^{k+1}-\bar{x}^k \rangle+\Psi(x^{k+1})-\Psi(\bar{x}^k)\\
&\!=\langle \nabla f(x^{k+1})-\nabla f(x^k),x^{k+1}-\bar{x}^k \rangle +\langle \nabla f(x^k), x^{k+1}-\bar{x}^k \rangle+\Psi(x^{k+1})-\Psi(\bar{x}^k) \\
&\! \leq \|\nabla f(x^{k+1})\!-\!\nabla f(x^k)\|_{W^{-1}}\|x^{k+1}\!-\!\bar{x}^k\|_{W}+\langle \nabla f(x^k), x^{k+1}\!-\!\bar{x}^k \rangle+\Psi(x^{k+1})\!-\!\Psi(\bar{x}^k) \\
&\! \overset{\eqref{gradient_lip}}{\leq} \|x^{k+1}\!-\!x^k\|_W\|x^{k+1}-\bar{x}^k\|_{W}+\langle \nabla f(x^k), x^{k+1}-\bar{x}^k \rangle+\Psi(x^{k+1})-\Psi(\bar{x}^k) \\
&\! \leq \|x^{k+1}\!-\!x^k\|_W^2 \!+\! \|x^{k+1}\!-\!x^k\|_W \|x^k\!-\!\bar{x}^k\|_{W} \!+\! \langle \nabla f(x^k), x^{k+1}\!-\!\bar{x}^k \rangle \!+\! \Psi(x^{k+1})\!-\!\Psi(\bar{x}^k).
\end{align*}
By taking expectation in both sides of the previous inequality we have:
\begin{align}
  \average [F(x^{k+1})]-F(\bar{x}^k) &\leq \average [\| x^{k+1}-x^k\|_W^2 ]  +  \average \left [ \|x^{k+1}\!-\!x^k\|_W \|x^k\!-\!\bar{x}^k\|_{W} \right ] \nonumber \\
&\quad +\average \left [ \langle \nabla f(x^k), x^{k+1}-\bar{x}^k \rangle+ \Psi(x^{k+1}) \right ]-\Psi(\bar{x}^k). \label{bound_opt_1}
\end{align}

\noindent From \eqref{bound_R} we have that $\|x^k-\bar{x}^k\|\leq R_W(x^0)$ and derive the following:
\begin{align*}
\average \left [ \|x^{k+1}\!-\!x^k\|_W \|x^k\!-\!\bar{x}^k\|_{W} \right ]&=\|x^k\!-\!\bar{x}^k\|_{W} \average \left [ \|x^{k+1}\!-\!x^k\|_W \right ] \nonumber \\
&\hspace{-0.1 cm} \overset{\eqref{x_diff_bounded2}}{\leq } c_\kappa(\tau) \sqrt{\average \left [ \|x^{k+1}\!-\!x^k\|_W^2 \right ]} \sqrt{\left (\average \left [ \|x^{k+1}\!-\!x^k\|_W \right ] \right )^2} \nonumber \\
& \leq c_\kappa(\tau) \; \average \left [ \|x^{k+1}\!-\!x^k\|_W^2 \right ],
\end{align*}
where the last step comes from Jensen's inequality. Thus, \eqref{bound_opt_1} becomes:
\begin{align}
  \average [F(x^{k+1})]-F(\bar{x}^k) &\leq  c_1(\tau) \average [\| x^{k+1}\!-\!x^k\|_W^2 ]\!+\! \average \left [ \left \langle \nabla f(x^k), x^{k+1}\!-\!\bar{x}^k \right \rangle \right ] \nonumber \\
&\quad + \average [\Psi(x^{k+1})] \!-\! \Psi(\bar{x}^k), \label{delta_2}
\end{align}
where $c_1(\tau)=\!\!\left (\! 1\!+\!c_\kappa(\tau) \right )$. We now explicitly express the second term in the right hand side of the above inequality:
\begin{align*}
 \average \left [ \langle \nabla f(x^k), x^{k+1}-\bar{x}^k \rangle \right ]& \overset{\eqref{iter_pcdm2}}{=} \average \left [ \langle \nabla f(x^k), x^k+T_{\coord^k}(x^k)-x^k_{\coord^k}-\bar{x}^k \rangle \right ] \\
& = \langle \nabla f(x^k),x^k-\bar{x}^k \rangle  + \average \left [ \langle \nabla f(x^k),T_{\coord^k}(x^k)-x^k_{\coord^k} \rangle \right ] \\
& \overset{\eqref{sep_prop2}}{=} \langle \nabla f(x^k),x^k-\bar{x}^k \rangle  + \frac{\tau}{N} \langle \nabla f(x^k),T_{[N]}(x^k)-x^k \rangle \\
& = \left ( 1-\frac{\tau}{N} \right ) \langle \nabla f(x^k),x^k-\bar{x}^k \rangle +  \frac{\tau}{N} \langle \nabla f(x^k),T_{[N]}(x^k)-\bar{x}^k \rangle
\end{align*}
So, by replacing it in \eqref{delta_2} and through \eqref{psi_separated} we get:
\begin{align}
   \average [F(x^{k+1})]\!-\!F(\bar{x}^k)  & \leq  c_1(\tau) \average [ \|x^{k+1}\!-\!x^k\|_W^2 ]+\frac{\tau}{N}  \langle \nabla f(x^k), T_{[N]}(x^k)\!-\!\bar{x}^k \rangle \label{interm_1} \\
& \quad + \left ( 1-\frac{\tau}{N} \right ) \langle \nabla f(x^k),x^k-\bar{x}^k \rangle \!+\! \frac{\tau}{N} \Psi(T_{[N]}(x^k)) \nonumber \\
& \quad  +\left (1- \frac{\tau}{N} \right ) \Psi (x^k)-\Psi(\bar{x}^k). \nonumber
\end{align}
By taking $y=\bar{x}^k$ and $x=x^k$ in \eqref{desc_lemma} we obtain:
\begin{equation*}
f(\bar{x}^k) \leq f(x^k) + \langle \nabla f(x^k),\bar{x}^k-x^k \rangle+\frac{1}{2} \|\bar{x}^k-x^k\|_{W}^2.
\end{equation*}
By rearranging this inequality, we obtain:
\begin{equation*}
 \langle \nabla f(x^k),x^k -\bar{x}^k \rangle \leq f(x^k) -f(\bar{x}^k)+\frac{1}{2} \|\bar{x}^k-x^k\|_{W}^2.
\end{equation*}
Through this and by rearranging terms in \eqref{interm_1}, we obtain:
\begin{align}
   \average [F(x^{k+1})]\!-\!F(\bar{x}^k)  & \leq  c_1(\tau) \average [ \|x^{k+1}\!-\!x^k\|_W^2 ]+ \left ( 1-\frac{\tau}{N} \right ) \left ( F(x^k)-F(\bar{x}^k) \right )  \label{interm_2} \\
& \quad + \frac{1}{2} \left ( 1-\frac{\tau}{N}  \right ) \|x^k-\bar{x}^k\|_W^2 \nonumber \\
& \quad + \frac{\tau}{N} \left (  \Psi(T_{[N]}(x^k))+\langle \nabla f(x^k), T_{[N]}(x^k)\!-\!\bar{x}^k \rangle -\Psi(\bar{x}^k)   \right).   \nonumber
\end{align}
Furthermore, from \eqref{x_diff_bounded2} we obtain:
\begin{align}
   \average [F(x^{k+1})]\!-\!F(\bar{x}^k)  & \leq  \left (c_1(\tau)+\frac{1}{2}\left ( 1-\frac{\tau}{N} \right )c_\kappa(\tau)^2 \right )\average [ \|x^{k+1}\!-\!x^k\|_W^2 ]  \label{interm_3}  \\
& \quad + \left ( 1-\frac{\tau}{N} \right ) \left ( F(x^k)-F(\bar{x}^k) \right )\nonumber  \\
& \quad + \frac{\tau}{N} \left (  \Psi(T_{[N]}(x^k))+\langle \nabla f(x^k), T_{[N]}(x^k)\!-\!\bar{x}^k \rangle -\Psi(\bar{x}^k)   \right).   \nonumber
\end{align}
Through the convexity of $\Psi(x)$ we have:
\begin{equation*}
 \Psi (\bar{x}^k) \geq \Psi (T_{[N]}(x^k))+ \left \langle \partial \Psi (T_{[N]}(x^k)), \bar{x}^k -T_{[N]}(x^k) \right \rangle
\end{equation*}
and by rearranging it we obtain:
\begin{equation*}
\Psi (T_{[N]}(x^k))- \Psi (\bar{x}^k) \leq \left \langle \partial \Psi (T_{[N]}(x^k)), T_{[N]}(x^k)-\bar{x}^k  \right \rangle.
\end{equation*}
From this and the optimality condition \eqref{prox_cond} and by replacing in \eqref{interm_3} we obtain:
\begin{align}
   \average [F(x^{k+1})]\!-\!F(\bar{x}^k)  & \leq  \left (c_1(\tau)+\frac{1}{2}\left ( 1-\frac{\tau}{N} \right )c_\kappa(\tau)^2 \right )\average [ \|x^{k+1}\!-\!x^k\|_W^2 ]  \label{interm_4} \\
& \quad + \left ( 1-\frac{\tau}{N} \right ) \left ( F(x^k)-F(\bar{x}^k) \right )  \nonumber \\
&  \quad + \frac{\tau}{N} \langle -W \left ( T_{[N]}(x^k)-x^k \right ), T_{[N]}(x^k)-\bar{x}^k \rangle.   \nonumber
\end{align}
Furthermore, by rearranging some terms and through the Cauchy-Schwartz inequality we obtain:
\begin{align*}
   \left \langle -W \left ( T_{[N]}(x^k)-x^k \right ), T_{[N]}(x^k)-\bar{x}^k \right \rangle & =   \left \langle -W \left ( T_{[N]}(x^k)-x^k \right ), T_{[N]}(x^k)-x^k+x^k- \bar{x}^k \right \rangle \\
& \leq \left \langle W \left ( T_{[N]}(x^k)-x^k \right ), \bar{x}^k-  {x}^k \right \rangle \\
& \leq \| W \left ( T_{[N]}(x^k)-x^k \right )\|_{W^{-1}} \|\bar{x}^k- {x}^k\|_W \\
& = \| T_{[N]}(x^k)-x^k \|_{W} \|\bar{x}^k-{x}^k\|_W.
\end{align*}
Now, recall that:
\begin{align*}
  \average \left [ \|x^{k+1}-x^k\|_W^2 \right ]&=\frac{\tau}{N} \|x^k-T_{[N]}(x^k) \|_W^2.
\end{align*}

\noindent Thus, from this and  \eqref{x_diff_bounded2} we get:
\begin{align*}
\frac{\tau}{N} \| T_{[N]}(x^k)-x^k \|_{W} \|\bar{x}^k-\bar{x}^k\|_W \leq c_\kappa(\tau) \sqrt{ \frac{\tau}{N}}   \average \left [ \|x^{k+1}-x^k\|_W^2 \right ].
\end{align*}

\noindent By replacing this in \eqref{interm_4} we obtain:
\begin{align}
   \average [F(x^{k+1})]\!-\!F(\bar{x}^k)  & \leq  \underbrace{\left (c_1(\tau)+\frac{1}{2}\left ( 1-\frac{\tau}{N} \right )c_\kappa(\tau)^2 +  c_\kappa(\tau) \sqrt{ \frac{\tau}{N}}\right )}_{c_2(\tau)} \average [ \|x^{k+1}\!-\!x^k\|_W^2 ] \nonumber \\
                  & \quad +  \left ( 1-\frac{\tau}{N} \right )  \left ( F(x^k)-F(\bar{x}^k) \right )  \nonumber \\
               & =  c_2(\tau) \average [ \|x^{k+1}\!-\!x^k\|_W^2 ] +  \left ( 1-\frac{\tau}{N} \right )  \left ( F(x^k)-F(\bar{x}^k) \right ).  \label{interm_5}
\end{align}
From \eqref{delta_norm_bound} we have:
\begin{equation*}
\average [ \|x^{k+1}\!-\!x^k\|_W^2 ] \leq -2 \average [ \Lambda^k].
\end{equation*}
Now, through this and by rearranging some terms in \eqref{interm_5} we obtain:
\begin{align*}
\frac{\tau}{N} \left ( \average [F(x^{k+1})]\!-\!F(\bar{x}^k) \right )  &  \leq   -2 c_2(\tau) \average [ \Lambda^k] +  \left ( 1-\frac{\tau}{N} \right )  \left ( F(x^k)-\average [ F(x^{k+1})] \right ).
\end{align*}
 Furthermore, from   \eqref{delta_bound} we obtain:
\begin{align*}
\average [F(x^{k+1})]\!-\!F(\bar{x}^k) & \leq \underbrace{\frac{N}{\tau} \left ( 2 c_2(\tau) + \left ( 1-\frac{\tau}{N} \right ) \right ) }_{c_3(\tau)} \left ( F(x^k)-\average [ F(x^{k+1})] \right ) \\
& = c_3(\tau) \left ( F(x^k)-\average [ F(x^{k+1})] \right ).
\end{align*}
By rearranging this inequality, we obtain:
\begin{align}
\average [F(x^{k+1})]\!-\!F(\bar{x}^k) & \leq  \frac{ c_3(\tau)}{1+ c_3(\tau)} \left ( F(x^k)-F(\bar{x}^k) \right ). \label{linear_main}
\end{align}
We denote $\theta=\frac{c_3(\tau)}{1+c_3(\tau)} <1$ and define
$\delta^k=F(x^{k+1}) -F(\bar{x}^k) $. By taking expectation over
$\eta^{k-1}$  in \eqref{linear_main} we arrive at:
\begin{equation*}
 \average [\delta^{k}] \leq \theta \average [\delta^{k-1}] \leq
 \dots \leq \theta^k \average [\delta^0],
\end{equation*}
and linear convergence is proved.
\end{proof}

\noindent Note that we have obtained  global linear convergence for
our distributed random coordinate descent method on the general
class of problems satisfying the generalized error bound property
(GEBP) given in \eqref{EEBF}, as opposed to the results in
\cite{TseYun:09,LuoTse:93} where the authors only show local linear
convergence for  deterministic  coordinate descent methods applied
to local error bound functions, i.e.  for all $k \geq k_0> 1$, where
$k_0$ is an iterate after which some error bound condition of the
form $\|x^k-\bar{x}^k\| \leq \bar{\kappa} \|\nabla^+ F(x^k)\|$ is
implicitly satisfied. In \cite{WanLin:13} global linear convergence
is also proved for the full gradient method but applied only to
problems having the error bound property where $\Psi$ is the set
indicator function of a polyhedron. Further, our results are more
general than the ones in
\cite{Nes:12,Nec:13,NecNes:11,NecClip:13a,NecPat:12,RicTak:12a},
where the authors prove linear convergence for the more restricted
class of problems having smooth and strongly convex objective
function. Moreover, our proof for convergence is  different from
those in these papers.

\noindent We now establish the number of iterations  $k^{\epsilon}_{\rho}$
which will ensure a $\epsilon$-suboptimal solution with probability
at least $1-\rho$.  In order to do so, we first recall that for constants $\epsilon>0$ and $\gamma \in (0,1)$ such that
$\delta^0 > \epsilon > 0$ and $k \geq \frac{1}{\gamma} \log \left (
\frac{\delta^0}{\epsilon}\right )$ we have:
 \begin{equation}
  \!\!  (1 \!-\! \gamma)^k \delta^0 =\left ( 1-\frac{1}{1/ \gamma}\right )^{(1 / \gamma)
   (\gamma k)} \delta^0 \leq \exp(-\gamma k) \delta^0 \leq \exp \left (-\log (\delta^0 / \epsilon)
    \right )\delta^0 =\epsilon.  \label{epsilon_relation}
 \end{equation}

%

\begin{corollary}\label{prob_result}
For a function $F$ satisfying Assumptions \ref{lip_fi} and the
generalized error bound property  \eqref{EEBF}, consider a probability
level $\rho \in (0,1)$, suboptimality $0< \epsilon < \delta_0$  and
an iteration counter:
\begin{equation*}
k^{\epsilon}_{\rho} \geq \frac{1}{1-\theta} \log \left ( \frac{\delta^0}{\epsilon \rho} \right ),
\end{equation*}
where recall that $\delta^0 = F(x^0) - F^*$ and $\theta$ is defined
in Theorem \ref{lin_converg}. Then, we have that the iterate
$x^{k^\epsilon_\rho}$ generated by \textbf{(P-RCD)}  method
satisfies:
\begin{equation}
 \mathbb{P} ( F(x^{k^\epsilon_\rho}) -F^* \leq \epsilon) \geq 1-\rho.
\end{equation}
\begin{proof}
 Under Theorem \ref{lin_converg} we have that:
\begin{equation*}
 \average \left  [ \delta^{k^\epsilon_\rho} \right ] \leq \theta^{k^\epsilon_\rho}
 \average \left [\delta^0 \right ]=\left (1-(1-\theta) \right )^{k^\epsilon_\rho}
 \average \left [ \delta^0 \right ]=\left (1-(1-\theta) \right )^{k^\epsilon_\rho}
 \delta^0  .
\end{equation*}
Through Markov's inequality and \eqref{epsilon_relation} we have that:
\begin{equation*}
 \mathbb{P}(\delta^{k^\epsilon_\rho} > \epsilon ) \leq
 \frac{\average [\delta^{k^\epsilon_\rho} ]}{\epsilon} \leq \frac{\left (1-
 (1-\theta) \right )^{k^\epsilon_\rho}}{\epsilon} \delta^0 \leq \rho.
\end{equation*}
and the proof is complete .
\end{proof}
\end{corollary}

\section{Conditions for generalized error bound functions}
\label{sec_gebfproperty}
In this section we investigate under which conditions a function $F$
satisfying  Assumption \ref{lip_fi} has the  generalized error bound
property defined in \eqref{EEBF} (see Definition~\ref{error_bound}).


\subsection{Case 1: $f$ strongly convex and $\Psi$ convex}
\label{case1_strong} We first  show that if $f$ satisfies Assumption
\ref{lip_fi} and additionally is also strongly convex, while $\Psi$
is a general convex function, then $F$ has the  generalized error
bound property defined in \eqref{EEBF}. Note that a similar result
was proved in \cite{TseYun:09}. For completeness, we also give the
proof. We consider $f$ to be strongly convex with constant
$\sigma_{W}$ w.r.t. the norm $\|\cdot\|_W$, i.e.:
\begin{equation}
\label{strongconvexity}
 f(y) \geq f(x)+\langle \nabla f(x),y-x \rangle+\frac{\sigma_W}{2} \|y-x\|_W^2.
\end{equation}
If $\Psi$ is strongly convex w.r.t. the norm $\|\cdot\|_W$, with
convexity parameter $\sigma_W^\Psi$, then we can redefine $f
\leftarrow f + \frac{\sigma_W^\Psi}{2} \| x - x^0\|_W^2$ and $\Psi
\leftarrow \Psi - \frac{\sigma_W^\Psi}{2} \| x - x^0\|_W^2$, so that
all the above assumptions hold for this new pair of functions.

\noindent By Fermat's rule \cite{RocWet:98} we have that
$T_{[N]}(x)$ is also the solution of the following problem:
\[ T_{[N]}(x) = \arg \min_y  \left \langle \nabla f(x) + W(T_{[N]}(x) - x), y-x
\right \rangle + \Psi(y)   \] and since $T_{[N]}(x)$ is optimal we get:
\begin{align*}
& \left \langle \nabla f(x) + W(T_{[N]}(x) -  x),  T_{[N]}(x) - x
\right \rangle + \Psi( T_{[N]}(x)) \\
& \qquad \leq \left \langle \nabla f(x) + W(T_{[N]}(x) - x), y - x
\right \rangle + \Psi(y) \quad \forall y.
\end{align*}
Since $f$ is strongly convex, then $X^*$ is a singleton and by
taking $y = \bar{x}$ we obtain:
\begin{align*}
& \left \langle \nabla f(x) + W(T_{[N]}(x) -  x), T_{[N]}(x) - x
\right \rangle + \Psi( T_{[N]}(x)) \\
& \qquad \leq \left \langle \nabla f(x) + W(T_{[N]}(x) - x), \bar{x}
 - x \right \rangle +  \Psi(\bar{x}).
\end{align*}
On the other hand from the optimality conditions for $\bar{x}$ and
convexity of $\Psi$ we get:
\[  \Psi(\bar{x}) + \langle \nabla f(\bar{x}), \bar{x} \rangle \leq \Psi(T_{[N]}(x)) +
\langle \nabla f(\bar{x}), T_{[N]}(x) \rangle.
\]
By adding  up the above two inequalities we obtain:
\begin{align*}
& \| T_{[N]}(x) -  x\|_W^2 + \left \langle \nabla f(x) - \nabla
f(\bar{x}), x - \bar{x} \right \rangle \\
& \qquad \leq  \left \langle \nabla f(\bar{x}) - \nabla f(x),
T_{[N]}(x) - x \right \rangle + \left \langle W(T_{[N]}(x) - x ),
\bar{x} - x \right \rangle.
\end{align*}
Now, from strong convexity \eqref{strongconvexity} and Lipschitz
continuity \eqref{gradient_lip} we get:
\begin{align*}
& \| T_{[N]}(x) -  x\|_W^2 +  \sigma_{W} \| x -\bar{x}\|_W^2  \leq 2 \|
x -\bar{x}\|_W   \; \|T_{[N]}(x) - x \|_W.
\end{align*}
Dividing  now both sides of this inequality by $\|x- \bar{x}\|_W$,
we obtain:
\begin{equation*}
\|x-\bar{x}\|_W \leq   \frac{2}{\sigma_W } \|\nabla^+F(x)\|_W,
\end{equation*}
i.e.  $\kappa_1=  \frac{2}{\sigma_W } $ and $\kappa_2=0$ in the Definition~\ref{error_bound}
of generalized error bound functions.


\subsection{Case 2: $\Psi$ indicator function of a  polyhedral set}
\label{sec_case2}
Another important category of problems \eqref{gen_form} that we
consider has the following objective function:
\begin{align}
\min_{x \in \rset^n } & \; F(x) \quad  \left (= \tilde{f} (P x) +
c^T x  +   \bo{I}_X(x) \right ), \label{problem1}
\end{align}
where $f(x) =\tilde{f}(Px) + c^T x$ is a smooth convex function, $P
\in \rset^{p \times n}$ is a constant matrix upon which we make no
assumptions and $\Psi(x)= \bo{I}_X(x)$ is the indicator function of
the polyhedral set $X$. Note that an objective function $F$ with the structure  in the form \eqref{problem1}  appears in many applications, see e.g.  the dual problem
\eqref{conv_conjug} obtained from the primal formulation \eqref{prob_sum} given  in Section \ref{sub_sec_motiv}.  Now, for proving the generalized error bound
property, we require that $f$ satisfies the following assumption:

\begin{assumption}
\label{ass_lip_conv}
We consider that $f(x)=\tilde{f}(Px)+c^T x$  satisfies Assumption \ref{lip_fi}.
We also assume that $\tilde{f}
(z)$ is strongly convex in $z$ with a constant $\sigma$ and  the
set of optimal solutions $X^*$ for problem \eqref{gen_form}    is
bounded.
\end{assumption}

\noindent  For problem \eqref{problem1}, functions $f$ under which
the set $X^*$ is bounded include e.g.  continuously differentiable
coercive functions \cite{MaZha:13}. Also, if \eqref{problem1} is a
dual formulation of   a primal problem \eqref{prob_sum} for which
the Slater condition holds, then by Theorem 1 of \cite{Man:85} we
have that the set of optimal Lagrange multipliers, i.e. $X^*$ in
this case, is compact.  Also, for $\Psi(x)= \bo{I}_X(x)$ we only
assume that $X$ is a polyhedron (possibly unbounded).

\noindent Our approach for proving the generalized  error bound
property is in a way similar to the one in
\cite{LuoTse:93,TseYun:09,WanLin:13}. However, our results are more
general in the sense that they hold globally,  while in
\cite{LuoTse:93,TseYun:09} the authors prove their results only
locally   and in the sense  that  we allow  the constraints set $X$
to be an unbounded polyhedron as opposed to the recent results in
\cite{WanLin:13} where the authors  show an error bound like
property  only for bounded polyhedra or for the entire space
$\rset^{n}$. This extension  is very important since it allows us
e.g.  to tackle the dual formulation of   a primal problem
\eqref{prob_sum} in which $X = \rset^n_+$ is the nonnegative orthant
and which appears in many practical applications. Last but not least
important is that our  error bound definition and gradient mapping
introduced  in this paper is more general than the one used in the
standard analysis of the classical error bound property (see e.g.
\cite{LuoTse:93,TseYun:09,WanLin:13}), as we can see from the
following example:

\begin{example} \label{exgeb}
Let us consider the following quadratic problem: $\min_{x \in
\rset^2_+} 1/2(x_1 - x_2)^2 +  x_1+x_2$.  We can easily see that
$X^* =\{0\}$  and thus this example satisfies Assumption
\ref{ass_lip_conv}.  Clearly, for this example the generalized error
bound property \eqref{EEBF} holds with e.g. $\kappa_1=\kappa_2 =1$.
However,  there is no finite constant $\kappa$  satisfying the
classical error bound property \cite{LuoTse:93,TseYun:09,WanLin:13}:
$\|x-\bar{x}\|_W \leq  \kappa  \|\nabla^+ F(x)\|_W$ for all $x \in
\rset^2_+$ (we can see this by taking  $x_1=x_2 \geq 1$ in the
previous inequality).
\end{example}

\noindent By definition, given that $\Psi(x)$ is a set  indicator function, we
observe that the gradient mapping of $F$ can be expressed in this
case as:
\begin{equation*}
 \nabla^+ F(x)=x-\Pi_X^W \left (x-W^{-1}\nabla f(x) \right ),
\end{equation*}
and  also note that $x^*$ is an optimal solution of \eqref{problem1} and of \eqref{gen_form} if and only if  $\nabla^+ F(x^*)=0$. The
following lemma establishes the Lipschitz continuity of $\nabla^+ F(x)$:
\begin{lemma}\label{prox_gradient_lip}
For a function $F$ whose smooth component satisfies Assumption
\ref{lip_fi}, we have that
\begin{equation}
\|\nabla^+ F(x)-\nabla ^+F(y)\|_W \leq 3\| x-y\|_W  \quad \forall x,
y \in X. \label{lip_proj_gradient}
\end{equation}
\end{lemma}

\begin{proof}
 By definition of $\nabla ^+ F(x)$ we have that:
\begin{align*}
&\|\nabla^+ F(x)\!-\!\nabla^+ F(y)\|_W \\
& \!=\! \|x \!-\! y+T_{[N]}(y)\!-\!T_{[N]}(x)\|_W \\
& \hspace{-0.2cm} \overset{\eqref{T_prox}}{\leq } \|x\!-\!y\|_W \!+\! \|\prox(x\!-\!W^{-1} \nabla f(x)) \!-\!\prox(y\!-\!W^{-1} \nabla f(y)) \|_W \\
& \hspace{-0.2cm} \overset{\eqref{prox_property1}}{\leq} \|x\!-\!y\|_W+\|x-y+W^{-1} (\nabla f(y)-\nabla f(x))\|_W \\
& \leq 2 \|x-y\|_W+\|\nabla f(x)-\nabla f(y)\|_{W^{-1}} \\
& \overset{\eqref{gradient_lip}}{\leq} 3\|x-y\|_W,
\end{align*}
and the proof is complete.
\end{proof}

\noindent The following lemma introduces an important  property for
the  operator $\Pi_X^W$.
\begin{lemma}\label{lemma_a1}
Given a convex set $X$, its projection operator $\Pi_X^W$  satisfies:
\begin{equation}
 \left \langle W \left ( \Pi_X^W(x)-x \right )  ,\Pi_X^W(x)-y \right \rangle \leq 0\quad \forall y \in X. \label{proj_properties}
\end{equation}
\end{lemma}
\begin{proof}
Following the definition of $\Pi_X^W$, we have that:
\begin{equation}
  \|x-\Pi_X^W(x)\|_W^2 \leq \|x-d\|_W^2\quad {\bl \forall d \in X}.  \label{norm_w_euclid}
\end{equation}
Since  $X$ is a convex set, consider a point:
\begin{equation*}
 d=\alpha y +(1-\alpha)\Pi_X^W(x) \in X \quad  \forall y \in X, \alpha \in [0,1],
\end{equation*}
and by \eqref{norm_w_euclid} we obtain:
\begin{equation*}
 \|x-\Pi_X^W(x)\|_W^2 \leq  \|x-(\alpha y +(1-\alpha)\Pi_X^W(x))\|_W^2.
\end{equation*}
If we elaborate the squared norms in the inequality above we arrive at:
\begin{equation*}
 0 \leq \alpha \left  \langle W\left ( \Pi_X^W(x)-x \right ),y-\Pi_X^W(x) \right \rangle +\frac{1}{2} \alpha^2 \|y-\Pi_X^W(x)\|^2.
\end{equation*}
If we divide both sides by $\alpha$ and let $\alpha \downarrow 0$, we get \eqref{proj_properties}.
\end{proof}

\noindent   The following lemma establishes an important property between
$\nabla f(x)$ and $\nabla ^+F(x)$.
\begin{lemma}\label{grad_proj_standard}
Given a function $f$ that satisfies \eqref{gradient_lip} and a convex set $X$, then the following inequality holds:
\begin{equation*}
 \left \langle \nabla f(x)-\nabla f(y), x-y  \right \rangle \leq 2 \| \nabla^+F(x)-\nabla^+F(y)\|_W \|x-y\|_W \quad \forall x,y \in X.
\end{equation*}
\end{lemma}
\begin{proof}
Denote $z=x-W^{-1}\nabla f(x)$, then by replacing $x=z$ and $y=\Pi_X^W(y-W^{-1}\nabla f(y)) $ in  Lemma \ref{lemma_a1} we obtain the following inequality:
\begin{equation*}
\left \langle W \left ( \Pi_X^W (z )- x \right) + \nabla f(x) , \Pi_X^W\!\! \left (z\right )\!-\!\Pi_X^W\!\! \left (y-W^{-1} \nabla f(y) \right ) \right \rangle \leq 0.
\end{equation*}
Through the definition of the projected gradient mapping, this
inequality can be rewritten as:
\begin{equation*}
 \left \langle \nabla f(x) - W \nabla^+F(x), x- \nabla^+F(x) -y+\nabla^+F(y) \right \rangle \leq 0.
\end{equation*}
If we further elaborate the inner product we obtain:
\begin{align}
\left \langle \nabla f(x), x-y \right \rangle & \leq \langle W \nabla^+F(x),x\!-\!y \rangle \!+\! \langle \nabla f(x),\nabla^+F(x) \!-\! \nabla^+F(y)  \rangle\label{inner_ineq} \\
&   \quad \!-\! \langle W \nabla^+F(x),\nabla ^+F(x) \!-\! \nabla^+F(y) \rangle. \nonumber
\end{align}
By adding two copies of \eqref{inner_ineq} with $x$ and $y$ interchanged we have:
\begin{align*}
& \left \langle \nabla f(x) - \nabla f(y), x - y \right \rangle \\
 & \leq \! \left \langle W (\nabla^+F(x) - \nabla^+F(y)), x - y\right \rangle +
  \left \langle \nabla f(x) - \nabla f(y), \nabla^+F(x) - \nabla^+F(y)\right \rangle \\
& \qquad - \|\nabla^+F(x) - \nabla^+F(y)\|_W^2 \\
& \leq  \left \langle W(\nabla^+F(x) - \nabla^+F(y)), x - y\right
\rangle  +  \left \langle \nabla f(x) - \nabla f(y), \nabla^+F(x)
 - \nabla^+F(y)\right \rangle.
\end{align*}
From this inequality, through Cauchy-Schwartz and \eqref{gradient_lip} we arrive at:
\begin{align*}
 \left \langle \nabla f(x) \!-\!\nabla f(y), x\!-\!y \right \rangle & \leq \|\nabla^+F(x) \!-\!\nabla^+F(y) \|_W \left (\|x-y\|_W+\|\nabla f(x)-\nabla f(y)\|_W^{-1} \right ) \\
& \leq 2 \| \nabla ^+ F(x) -\nabla ^+ F(y) \|_W \|x-y\|_W,
\end{align*}
and the proof is complete.
\end{proof}

\noindent  We now introduce the following lemma regarding the
optimal set $X^*$,  see also \cite{LuoTse:93, WanLin:13}.
\begin{lemma}\label{eq_const_opt}
Under Assumption \ref{ass_lip_conv}, there exists a unique $z^*$
such that:
\begin{equation*}
 P x^*=z^*\quad \forall x^* \in X^*,
\end{equation*}
and furthermore:
\begin{equation*}
 \nabla f(x)=P^T \nabla \tilde{f}(z^*)+c
\end{equation*}
is  constant for all $x \in Q=\left \{y  \in X: \; Py=z^* \right
\}$.
\end{lemma}
\begin{proof}
 Given that $f(x)$ as defined in problem \eqref{problem1} is a convex function, then for any two optimal solutions $x_1^*, \; x_2^* \in X^*$ we obtain:
\begin{align*}
 f\left ((x_1^*+x_2^*)/2 \right )=\frac{1}{2} \left ( f(x_1^*)+f(x_2^*) \right ),
\end{align*}
which by the definition of $f$ is equivalent to:
\begin{align*}
 \tilde{f}\left ( (P x_1^*+Px_2^*)/2 \right ) +\frac{1}{2} c^T (x_1^*+x_2^*)=\frac{1}{2} \left ( \tilde{f}(Px_1^*)+\tilde{f}(Px_2^*)+c^T(x_1^*+x_2^*) \right ).
\end{align*}
If we substract $c^T(x_1^*+x_2^*)$ in both sides and by the strong
convexity of $\tilde{f}$ we have that $P x_1^*=Px_2^*$. Thus,
$z^*=Px^*$ is unique. From this, it is straightforward to see that $
\nabla f(x)=P^T \nabla \tilde{f}(z^*)+c$ is constant for all $x \in
Q$.
\end{proof}

\noindent  Consider now a point $x \in X$ and denote by $\bo{q}=\Pi^W_Q(x)$ the
projection of the point $x$ onto the set $Q=\left \{y  \in X: \; Py=z^* \right
\}$, as defined in Lemma
\ref{eq_const_opt}, and by $\bar{\bo{q}}$ its projection onto  the
optimal set $X^*$, i.e. $\bar{\bo{q}}=\Pi_{X^*}^W (\bo{q})$. Given
the set $Q$, the distance to the optimal set can be decomposed as:
\begin{equation*}
\|x-\bar{x}\|_W \leq \|x-\bar{\bo{q}}\|_W \leq
\|x-\bo{q}\|_W+\|\bo{q}  -\bar{\bo{q}}\|_W.
\end{equation*}
Given this inequality, the outline for proving the  generalized
error bound property (GEBP) from \eqref{EEBF} in this case is to
obtain appropriate upper bounds for $\|x-\bo{q}\|_W$ and $\|\bo{q}
-\bar{\bo{q}}\|_W$.  In the sequel we introduce lemmas for
establishing bounds for these two terms.
\begin{lemma}\label{Q_x_bound}
 Under Assumption \ref{ass_lip_conv}, there exists a constant $\gamma_1$ such that:
\begin{equation*}
 \|x-\bo{q}\|_W^2 \leq \gamma_1^2 \frac{2}{\sigma} \|\nabla^+F(x)\|_W \|x-\bar{x}\|_W \quad \forall x \in X.
\end{equation*}
\end{lemma}
\begin{proof}
Corollary 2.2 in \cite{Rob:73} states that if we have the following
two sets of constraints:
\begin{align}
 A y \leq b_1, \; P y=d_1 \label{constraint1} \\
A y \leq b_2, \; P y=d_2, \label{constraint2}
\end{align}
then there exists a finite constant $\gamma_1$  such that for a point $y_1$
which satisfies the first set of constraints and a point $y_2$ which
satisfies  the second one we have:
\begin{equation}
 \|y_1-y_2\|_W \leq \gamma_1 \left \|\begin{matrix}
 \Pi_{\rset^+} (b_1 -b_2) \\ d_1-d_2 \end{matrix} \right \|_W. \label{constraint_bound}
\end{equation}
Furthermore,  $\gamma_1$ is only dependent on the
matrices $A$ and $P$ (see \cite{Rob:73} for more details).
Given that $X$ is polyhedral, we can express it as $X=\left \{x \in
\rset^n : A x \leq b \right \}$. Thus, for $x \in X$, we can take
$b_1=b$, $d_1=Px$ in \eqref{constraint1}, and $b_2=b$, $d_2=z^*$ in
\eqref{constraint2} such that:
\begin{align}
 A y \leq b, \; P y=Px \label{constraint1_1} \\
A y \leq b, \; P y=z^* .\label{constraint2_1}
\end{align}

\noindent  Evidently, a point $x \in X$ is feasible for \eqref{constraint1_1}. Consider now a point $y_2$ feasible for
\eqref{constraint2_1}. Therefore, from \eqref{constraint_bound} there exists a constant $\gamma_1$ such that:
\begin{equation*}
\|x-y_2\|_W \leq \gamma_1 \|Px-z^*\|_W \quad \forall x \in X.
\end{equation*}
Furthermore, from the definition of $\bo{q}$ we get:
\begin{equation}
 \|x-\bo{q}\|_W^2 \leq \|x-y_2\|_W^2  \leq  \gamma_1^2
  \|Px-z^*\|_W^2  \quad \forall x \in X. \label{Q_distance}
\end{equation}
From the strong convexity of $\tilde{f}(z)$ we have the following property:
\begin{equation*}
 \sigma \|Px-z^*\|_W^2 \leq \left \langle \nabla \tilde{f}(Px)-\nabla
  \tilde{f}(P\bar{x}), Px-P \bar{x} \right \rangle=
  \left \langle \nabla f(x)-\nabla f(\bar{x}), x-\bar{x} \right \rangle
\end{equation*}
for all $\bar{x} \in X^*$. From this inequality and Lemma
\ref{grad_proj_standard} we obtain:
\begin{equation*}
 \sigma \|Px-z^*\|_W^2 \leq 2 \| \nabla^+F(x)- \nabla^+F(\bar{x}) \|_W \|x-\bar{x}\|_W.
\end{equation*}
Since $\bar{x} \in X^*$, it is well known that $\nabla^+F(\bar{x})=0$. Thus, from this and \eqref{Q_distance} we get:
\begin{equation*}
 \|x-\bo{q}\|_W^2  \leq   \gamma_1^2 \frac{2}{\sigma} \|\nabla^+F(x)\|_W \|x-\bar{x}\|_W
\end{equation*}
and the proof is complete.
\end{proof}

\noindent   Note that, if in \eqref{problem1} we have $c=0$, then by definition
we have that $Q=X^*$, and thus  the term $\|\bo{q}
-\bar{\bo{q}}\|_W=0$. In such a case, also note that
$\bo{q}=\bar{x}$ and through the previous lemma, in which we
established an upper bound for $\|x-\bo{q}\|_W$, we can prove outright
the error bound property \eqref{EEBF} with $\kappa_1 =\gamma_1^2
\frac{2}{\sigma}$ and $\kappa_2=0$. If $c \not = 0$, the following
two lemmas are introduced to investigate the distance between
a point and a solution set of a linear programming problem and then
to establish a bound for $\|\bo{q}-\bar{\bo{q}}\|_W$.

\begin{lemma}
\label{lemma_lip_bound}
Consider an LP on a nonempty polyhedral set $Y$:
\begin{align}
 \min_{y \in Y} b^T y \label{lp_gen},
\end{align}
and assume that the optimal set $Y^* \subseteq Y$ is nonempty,
convex and bounded. Let $\bar{y}$ be the projection of a point $y
\in Y$ on the optimal set $Y^*$. For this problem we have that:
\begin{equation}
 \|y-\bar{y}\|_W \leq \gamma_2 \left ( \|y-\bar{y}\|_W+\|b\|_{W^{-1}} \right )
 \|y-\Pi_Z^W (y-W^{-1}b)\|_W \quad \forall y \in Y \label{bound_lip_lemma},
\end{equation}
where $Z$ is any closed convex set satisfying $Y \subseteq Z$ and $\gamma_2$ depends on $Y$ and $b$.
\end{lemma}
\begin{proof}
 Because the solution set $Y^*$ is nonempty, convex and bounded, then the linear program
 \eqref{lp_gen} is equivalent to the following problem:
\begin{align*}
  \min_{y \in Y^*} b^T y,
\end{align*}
and as a result, the linear program \eqref{lp_gen} is solvable. Now,
by the duality theorem of linear programming, the dual problem of
\eqref{lp_gen}:
\begin{equation}
 \max_{\mu \in Y'} l(\mu) \label{lp_dual},
\end{equation}
 is well defined, solvable and strong duality holds, where $Y' \subseteq \rset^m$ is the dual feasible set.  For any pair of primal-dual
feasible points $(y,\mu)$ for problems \eqref{lp_gen} and
\eqref{lp_dual}, we have a corresponding pair of optimal solutions
$(y^*,\mu^*)$. By the solvability of \eqref{lp_gen} we have from
Theorem 2  of \cite{Rob:73}, that there exists a constant $\gamma_2$
depending on $Y$ and $b$ such that we have the bound:
\begin{equation*}
 \left \| \begin{matrix} y-y^* \\ \mu-\mu^* \end{matrix}
 \right \|_{\text{diag}(W,I_m)} \leq \gamma_2 |b^Ty - l(\mu)|.
\end{equation*}

\noindent By strong duality, we have that $l(\mu^*) = b^T \bar{y}$. Thus,
taking $\mu=\mu^*$ and through the optimality conditions of
\eqref{lp_gen} we obtain:
\begin{equation*}
\|y-y^*\|_W \leq \gamma_2 \langle b,y-\bar{y} \rangle.
\end{equation*}
From this inequality and $\|y-\bar{y}\|_W \leq \|y-y^*\|_W$ we arrive at:
\begin{equation}
\|y-\bar{y}\|_W \leq \gamma_2  \langle b,y-\bar{y} \rangle.
\label{bound_inter}
\end{equation}
By Lemma \ref{lemma_a1}, we have that:
\begin{equation*}
 \left \langle  W \left ( \Pi_Z^W\left (y-W^{-1} b \right)-\left (y-W^{-1} b
  \right) \right ),  \Pi_Z^W(y-W^{-1} b)-\bar{y} \right \rangle \leq
  0.
\end{equation*}
This inequality can be rewritten as:
\begin{align*}
 \left \langle b, y-\bar{y} \right \rangle &\leq \left \langle  W\left ( y-\Pi_Z^W(y-W^{-1} b) \right ) , y-\bar{y}+W^{-1} b+\Pi_Z^W(y-W^{-1}b)-y \right \rangle \\
& \leq \left \langle W \left ( y-\Pi_Z^W(y-W^{-1} b) \right ) , y-\bar{y}+W^{-1} b \right \rangle \\
& \leq \|y-\Pi_Z^W(y-W^{-1} b)\|_W  \left (
\|y-\bar{y}\|_W+\|b\|_{W^{-1}} \right ).
\end{align*}
From this inequality and \eqref{bound_inter} we obtain:
\begin{equation*}
 \|y-\bar{y}\|_W \leq \gamma_2  \left (\|y-\bar{y}\|_W+\|b\|_{W^{-1}} \right )\|y-\Pi_Z^W(y-W^{-1}b) \|_W
\end{equation*}
and the proof is complete.
\end{proof}

\begin{lemma}\label{Q_bound}
If Assumption \ref{ass_lip_conv} holds for optimization problem
\eqref{problem1}, then there exists a constant $\gamma_2$ such that:
\begin{equation}
\|\bo{q}-\bar{\bo{q}}\|_W \leq \gamma_2 \left
(\|\bo{q}-\bar{\bo{q}}\|_W+\|\nabla f(\bar{x})\|_{W^{-1}} \right)
\|\nabla F^+(\bo{q})\|_W \quad \forall x \in X. \label{bar_x_bound}
\end{equation}
\end{lemma}
\begin{proof}
By Lemma \ref{eq_const_opt}, we have that $Px=z^*$ for all $x \in
Q$. As a result, the following optimization problem:
\begin{equation*}
 \min_{x \in Q} \tilde{f}(z^*)+c^T x
\end{equation*}
has the same solution set as problem \eqref{problem1}, due to the
fact that $X^* \subseteq Q \subseteq X$. Since $z^*$ is a constant,
then we can formulate the equivalent problem:
\begin{equation*}
 \min_{x \in Q} \nabla f(\bar{x})^T x \qquad  \left (= \nabla \tilde{f}(z^*)^T z^*+c^T x \right
 ).
\end{equation*}
Note that $\nabla f(\bar{x})=P^T \nabla \tilde{f}(z^*)+c$ is
constant and under Assumption \ref{ass_lip_conv} we have  that $X^*$
is convex and  bounded.  Furthermore, since $\bar{x},\bo{q} \in Q$,
then $\nabla f(\bar{x})=\nabla f(\bo{q})$. Considering these
details, and by taking $Y=Q$, $Z=X$, $y=\bo{q}$ and $b=\nabla
f(\bar{x})$ in Lemma
 \ref{lemma_lip_bound} and applying it to the previous problem, we obtain \eqref{bar_x_bound}.
\end{proof}

\noindent   The next theorem establishes the generalized error bound property
for optimization problems in the form \eqref{problem1} having
objective functions  satisfying Assumption~\ref{ass_lip_conv}.
\begin{theorem}\label{EEBF_euclid}
Under Assumption \ref{ass_lip_conv}, the function $F(x)= \tilde{f}
(P x) + c^T x  +   \bo{I}_X(x)$ satisfies the following global generalized
error bound property:
\begin{equation}
 \|x-\bar{x}\|_W \leq \left (\kappa_1+\kappa_2 \|x-\bar{x}\|_W^2 \right ) \|\nabla^+F(x) \|_W \quad \forall  x \in
 X,
\end{equation}
where $\kappa_1$  and $\kappa_2$ are two nonnegative constants.
\end{theorem}

\begin{proof}
 Given that $\bar{x} \in X^*$, it is well known that $\nabla^+F(\bar{x})=0$ and by Lemma \ref{prox_gradient_lip} we have:
\begin{equation*}
 \| \nabla^+F(x) \|_W= \|\nabla^+F(x) -\nabla^+F(\bar{x}) \|_W \leq  3 \|x-\bar{x}
 \|_W.
\end{equation*}
From this inequality and by applying Lemma \ref{prox_gradient_lip},
we also have:
\begin{align*}
 \| \nabla^+F(\bo{q}) \|_W ^2 & \leq \left ( \| \nabla^+F(x)\|_W  +\| \nabla^+F(\bo{q}) -\nabla^+F(x) \| _W\right )^2 \\
& \leq 2 \| \nabla^+F(x) \|_W^2 +2 \| \nabla^+F(\bo{q}) -\nabla^+F(x)\|_W^2 \\
& \leq 6 \left ( \| \nabla^+F(x) \|_W \|x-\bar{x}\|_W +3 \|\bo{q}- x\|^2 \right ).
\end{align*}
From this and Lemma \ref{Q_bound}, we arrive at the following:
\begin{align}
\label{bounding_xQ}  \|\bo{q}\!-\!\bar{\bo{q}}\|_W^2  &\leq \gamma_2^2 \left (\|\bo{q}\!-\!\bar{\bo{q}}\|_W+\|\nabla f(\bar{x})\|_{W^{-1}} \right)^2 \|\nabla F^+(\bo{q})\|_W^2   \\
& \leq 6 \gamma_2^2 \left (\|\bo{q}\!-\!\bar{\bo{q}}\|_W+\|\nabla f(\bar{x})\|_{W^{-1}} \right)^2 \left (  \| \nabla^+F(x) \|_W \|x\!-\!\bar{x}\|_W +3 \|\bo{q}\!-\! x\|^2 \right ). \nonumber
\end{align}
Note that since $X^*$ is a bounded set,  then we can imply the
following upper bound:
\begin{equation*}
 \|\nabla f(\bar{x})\|_{W^{-1}} \leq \beta  \qquad  \left (=\max_{x^* \in X^*} \|\nabla f(x^*)\|_{W^{-1}} \right ).
\end{equation*}

\noindent   Furthermore, $\bar{\bo{q}} \in Q$ since $X^* \subseteq Q$. From this and through the nonexpansive property of the projection operator we obtain:
\begin{align*}
 \|\bo{q}-\bar{\bo{q}}\|_W &\leq \|\bo{q}-\bar{x}\|_W\!+\!\|\bar{x}-\bar{\bo{q}}\|_W \leq \|x-\bar{x}\|_W\!+\!\|\bar{x}\!-\!\bar{\bo{q}}\|_W\leq \|x-\bar{x}\|_W+\|x-\bo{q}\|_W \\
& \leq \|x-\bar{x}\|_W+\|\bar{x}-\bo{q}\|_W \leq 3 \|x-\bar{x}\|_W.
\end{align*}

\noindent   From this and \eqref{bounding_xQ} we get the following bound:
\begin{align}
& \|\bo{q}\!-\!\bar{\bo{q}}\|_W^2 \label{bounding_xQ_2} \\
& \leq 6 \gamma_2^2 (3\|x-\bar{x}\|_W+\beta)^2 \left ( \| \nabla^+F(x) \|_W \|x\!-\!\bar{x}\|_W +3 \|\bo{q}\!-\! x\|_W^2 \right ) \nonumber \\
& \leq 6 \gamma_2^2 (18 \|x-\bar{x}\|_W^2+2 \beta^2) \left ( \| \nabla^+F(x) \|_W \|x\!-\!\bar{x}\|_W +3 \|\bo{q}\!-\! x\|_W^2 \right ). \nonumber
\end{align}
Given the definition of $\bar{x}$ we have that:
\begin{equation*}
 \|x-\bar{x}\|_W^2 \leq \|x-\bar{\bo{q}}\|_W^2 \leq \left ( \|x-\bo{q}\|_W +\|\bo{q}-\bar{\bo{q}}\|_W \right )^2 \leq 2 \|x-\bo{q}\|_W^2+2\|\bo{q}-\bar{\bo{q}}\|_W^2.
\end{equation*}
From Lemma \ref{Q_x_bound} and \eqref{bounding_xQ_2}, we can establish an upper bound for the right hand side of the above inequality:
\begin{equation}
 \|x-\bar{x}\|_W^2 \leq  (\kappa_1+\kappa_2 \|x-\bar{x}\|_W^2) \|\nabla^+F(x) \|_W \|x-\bar{x}\|_W, \label{divide_optim}
\end{equation}
 where:
\begin{align*}
 \kappa_1 & =24 \gamma_2^2 \beta^2 \left (1+\frac{6\gamma_1^2}{\sigma} \right ) +\frac{4\gamma_1^2}{\sigma} \\
  \kappa_2 & = 256 \gamma_2^2 \left ( 1+\frac{6\gamma_1^2}{\sigma} \right ).
\end{align*}
If we divide both sides of \eqref{divide_optim} by $\|x-\bar{x}\|_W$, the proof is complete .
\end{proof}


\subsection{Case 3: $\Psi$  polyhedral function}
We now consider general optimization problems of the form:
\begin{align}
\min_{x \in \rset^n } & \; F(x) \qquad  \left( = \tilde{f} (P x) +
c^T x + \Psi(x) \right), \label{problem1_psi}
\end{align}
where $\Psi(x)$ is a polyhedral function. A function $\Psi :\rset^n
\rightarrow \rset^{}$ is polyhedral if its epigraph, $\text{epi} \;
\Psi=\left \{ (x,\zeta): \Psi(x) \leq \zeta \right \}$, is a
polyhedral set. There are numerous functions $\Psi$ which are polyhedral,
e.g. $\bo{I}_X(x)$ with $X$  a polyhedral set,  $\|x\|_1$,
$\|x\|_{\infty}$ or combinations of these functions. Note that an objective function with the structure
\eqref{problem1_psi} appears in many applications (see e.g. the
constrained Lasso problem \eqref{sol_norm} in Section
\ref{sub_sec_motiv}). Now, for proving the generalized  error bound
property, we require that $F$ satisfies the following assumption.

\begin{assumption}
\label{ass_lip_conv_Psi} We consider that $f(x)=\tilde{f}(Px)+c^T x$   satisfies Assumption \ref{lip_fi}. Further, we
assume that $\tilde{f} (z)$ is strongly convex in $z$ with a
constant $\sigma$ and   the optimal set $X^*$ is bounded. We also
assume  that $\Psi(x)$ is bounded above on its domain by a finite
value $\bar{\Psi} < \infty$, i.e. $\Psi(x) \leq \bar{\Psi}$ for all $x \in \text{dom} \;
\Psi$,  and is Lipschitz continuous w.r.t. norm  $\|\cdot\|_W$ with
a constant $L_\Psi$.
\end{assumption}

\noindent The proof of the generalized error bound property under
Assumption \ref{ass_lip_conv_Psi} is similar to that of
\cite{TseYun:09}, but it requires new proof ideas and is done under
different assumptions, e.g. that $\Psi(x)$  is bounded above on its
domain. Boundedness of $\Psi$ is in practical applications   usually
not  restrictive. Since $\Psi(x) \leq \bar{\Psi}$ is satisfied for
any $x \in \text{dom} \; \Psi$, then problem $\eqref{problem1_psi}$
is equivalent to the following one:
\begin{align*}
\min_{x \in \rset^n}  & \quad f(x) + \Psi(x)   \\
\text{s.t.}  & \quad \Psi(x) \leq \bar{\Psi}.
\end{align*}
Consider now an additional variable $\zeta \in \rset^{}$. Then, the
previous problem is equivalent to the following problem:
\begin{align}
\min_{x \in \rset^n,\zeta \in \rset }  & \quad f(x) + \zeta  \label{problem2_psi} \\
\text{s.t.} \quad &  \Psi(x) \leq \zeta, \; \Psi(x) \leq \bar{\Psi}.
\nonumber
\end{align}
Take an optimal pair $(x^*,\zeta^*)$ for problem
\eqref{problem2_psi}. We now prove that $\zeta^*=\Psi(x^*)$.
Consider that $(x^*,\zeta^*)$ is strictly feasible, i.e. $\Psi(x^*)
< \zeta^*$. Then, we can imply that $(x^*, \Psi(x^*))$ is feasible
for \eqref{problem2_psi} and the following inequality holds:
\begin{equation*}
f( x^*) + \Psi(x^*) < f(x^*) + \zeta^*,
\end{equation*}
which  contradicts the fact that $(x^*,\zeta^*)$ is optimal.
Thus, it remains that $\Psi(x^*)=\zeta^*$.

\noindent The following lemma establishes an equivalence between
\eqref{problem2_psi} and another problem:
\begin{lemma}\label{problem_equiv}
Under Assumption \ref{ass_lip_conv_Psi}, the following problem is
equivalent to \eqref{problem2_psi}:
\begin{align}
\min_{x \in \rset^n,\zeta \in \rset }  & \quad f(x) +\zeta  \label{problem3_psi} \\
\text{s.t.}  \quad &  \Psi(x) \leq \zeta, \; \zeta \leq \bar{\Psi}.
\nonumber
\end{align}
\end{lemma}

\begin{proof}
The proof of this lemma consists of the following two stages: we
prove  that an optimal point of \eqref{problem2_psi} is an optimal
point of \eqref{problem3_psi}, and then we prove its converse.
Consider now an optimal pair $(x^*,\zeta^*)$ for
\eqref{problem2_psi}. Since $(x^*,\zeta^*)$ is feasible for
\eqref{problem2_psi}, we have that $\Psi(x^*) \leq \zeta^*$ and
$\Psi(x^*) \leq \bar{\Psi}$. Recall that $\Psi(x^*)=\zeta^*$. Then,
 $\zeta^* \leq \bar{\Psi}$ and thus  $(x^*,\zeta^*)$ is feasible for
\eqref{problem3_psi}. Assume now  that  $(x^*,\zeta^*)$  is not
optimal for \eqref{problem3_psi}. Then, there exists an optimal pair
$(\tilde{x}^*,\tilde{\zeta}^*)$ of \eqref{problem3_psi} such that:
\begin{equation}
f(\tilde{x}^*) + \tilde{\zeta}^* < f(x^*)+\zeta^*. \label{optim1}
\end{equation}
Since $(\tilde{x}^*,\tilde{\zeta}^*)$ is feasible for
 \eqref{problem3_psi}, we have that $\tilde{x}^* \leq
 \tilde{\zeta}^*$ and inherently $\Psi(\tilde{x}^*) \leq
 \bar{\Psi}$. Thus, $(\tilde{x}^*,\tilde{\zeta}^*)$   is feasible and from \eqref{optim1} note that it is optimal for problem \eqref{problem2_psi},
 which contradicts the fact that $(x^*,\zeta^*)$ is optimal for
\eqref{problem2_psi}.

\noindent Consider now the converse. That is,  there exists a pair
$(\tilde{x}^*,\tilde{\zeta}^*)$ which is optimal for
\eqref{problem3_psi} and  is not optimal for
\eqref{problem2_psi}. Following the same lines as before, note that
$(\tilde{x}^*,\tilde{\zeta}^*)$ is feasible for
\eqref{problem2_psi}. Assume now that
$(\tilde{x}^*,\tilde{\zeta}^*)$ is not optimal for
\eqref{problem2_psi}. Then, there  exists  a pair $(x^*,\zeta^*)$
such that:
\begin{equation}
 f(x^*)+\zeta^*< f(\tilde{x}^*) + \tilde{\zeta}^*.  \label{optim2}
\end{equation}
Since $({x}^*,{\zeta}^*)$ is feasible for \eqref{problem2_psi},
recall that it is also  feasible for \eqref{problem3_psi}.  Thus,
$(x^*,\zeta^*)$ is feasible and optimal for
\eqref{problem3_psi}, which contradicts the fact that $(\tilde{x}^*,\tilde{\zeta}^*)$ is optimal for  \eqref{problem3_psi}.
\end{proof}

\noindent  Now, if we denote $z=[x^T \; \zeta]^T$, then problem
\eqref{problem3_psi} can be rewritten as:
\begin{align}
\min_{z \in \rset^{n+1} }  & \quad \tilde{F}(z)  \quad
 \left( = \tilde{f} (\tilde {P} z) + \tilde{c}^T z \right) \label{problem4_psi}    \\
\text{s.t.}  \quad &  z \in Z \nonumber,
\end{align}
where $\tilde{P}=[P \;\; 0]$  and $\tilde{c}=[c^T \; 1]^T$. The
constraint set for this problem is: \[ Z=\left \{z=[x^T \; \zeta]^T :
z \in \text{epi} \; \Psi, \zeta \leq \bar{\Psi}  \right \}.  \]
Recall
that from Assumption \ref{ass_lip_conv_Psi} we have that $\text{epi}
\; \Psi$ is polyhedral, i.e. there exists a matrix $C$ and a vector
$d$ such that we can express $\text{epi} \; \Psi= \left \{ (x,\zeta): C
[x^T \; \zeta]^T \leq d  \right \}$. Thus, we can write the
constraint set $Z$ as:
\begin{equation*}
 Z=\left \{z=[x^T \; \zeta]^T : \begin{bmatrix} C \\ e_{n+1}^T \end{bmatrix} z \leq \begin{bmatrix} d \\ \bar{\Psi} \end{bmatrix} \right \},
\end{equation*}
i.e. $Z$ is polyhedral.  Denote by $Z^*$ the set of optimal points
of problem \eqref{problem3_psi}. Then, from $X^*$ being bounded  in
accordance with Assumption \ref{ass_lip_conv_Psi}, and the fact that
$ \Psi(x^*) =\zeta^*$, with $\Psi$ continuous function, it can be
observed that $Z^*$ is also bounded. We now denote
$\bar{z}=\Pi^{\tilde{W}}_{Z^*}(z)$, where
$\tilde{W}=\text{diag}(W,1)$. Since by Lemma \ref{problem_equiv} we
have that problems \eqref{problem2_psi} and \eqref{problem4_psi} are
equivalent, then we can apply the theory of the previous subsection
to problem \eqref{problem4_psi}. That is, we can find two
nonnegative constants $\tilde{\kappa_1}$ and $\tilde{\kappa_2}$ such
that:
\begin{equation}
\|z-\bar{z}\|_{\tilde{W}} \leq \left ( \tilde{\kappa}_1+\tilde{\kappa}_2 \|z-\bar{z}\|_{\tilde{W}^2} \right ) \|\nabla^+ \tilde{F}(z)\|_{\tilde{W}}  \quad \forall z \in  Z.  \label{bound_in_z}
\end{equation}
The proximal gradient mapping in this case, $\nabla^+ \tilde{F} (z)$ is defined as:
 \begin{equation*}
 \nabla^+ \tilde{F}(z)=z-\Pi_Z^{\tilde{W}} \left (z-\tilde{W}^{-1}\nabla \tilde{F}(z) \right ),
\end{equation*}
where the projection operator $\Pi_Z^{\tilde{W}}$ is defined in the same manner as $\Pi_X^W$.
We now show that from the error bound inequality \eqref{bound_in_z} we can derive an error bound inequality for problem \eqref{problem1_psi}.
From the definitions of $z$, $\bar{z}$ and $\tilde{W}$,  we derive the following lower bound for the term on the right-hand side:
\begin{equation}
 \|z-\bar{z}\|_{\tilde{W}} = \left \| \begin{matrix} x-\bar{x} \\ \zeta-\bar{\zeta} \end{matrix} \right \|_{\tilde{W}} \geq \|x-\bar{x}\|_{W} .
\end{equation}
Further, note that we can express:
\begin{equation}
\|z-\bar{z}\|^2_{\tilde{W}}=\|x-\bar{x}\|^2_{W}
+(\zeta-\bar{\zeta})^2=\|x-\bar{x}\|^2_{W} +|\zeta-\bar{\zeta}|^2.
\label{z_ex_squared}
\end{equation}
Now, if $\zeta \leq \bar{\zeta}$, then from
$\bar{\zeta}=\Psi(\bar{x})$ and the Lipschitz continuity of $\Psi$
we have that:
\begin{equation*}
 |\zeta-\bar{\zeta}| = \bar{\zeta}-\zeta \leq \Psi(\bar{x})-\Psi(x) \leq L_{\Psi} \|x-\bar{x} \|_W .
\end{equation*}
Otherwise, if $\zeta > \bar{\zeta}$, we have that:
\begin{equation*}
  |\zeta-\bar{\zeta}|=\zeta-\bar{\zeta}  \leq \bar{\Psi} -\bar{\zeta} \leq  |\bar{\Psi}| + |\bar{\zeta}| \overset{\Delta}{=} \kappa_1'.
\end{equation*}
From these two inequalities we derive the following inequality for $| \zeta-\bar{\zeta} |^2$:
\begin{align*}
 |\zeta-\bar{\zeta}|^2 & \leq (\kappa_1' +L_{\Psi} \|x-\bar{x}\|_W)^2 \leq
 2 \kappa_1'^2+2 L_{\Psi}^2 \|x-\bar{x}\|_W^2.
\end{align*}
Therefore, the following upper bound for $\|z-\bar{z}\|^2_{\tilde{W}}$ is established:
\begin{equation}
 \|z-\bar{z}\|^2_{\tilde{W}} \leq  2 \kappa_1'^2 + (2 L_{\Psi}^2+1) \|x-\bar{x}\|^2_W. \label{bound_squared}
\end{equation}

\noindent   We are now ready to present the main result of this section that shows the generalized error bound property for problems in the form \eqref{problem1_psi} under  general polyhedral  $\Psi$:
\begin{theorem}
\label{theorem_generalpsi}
Under Assumption \ref{ass_lip_conv_Psi}, the function $F(x)= \tilde{f}
(P x) + c^T x  +   \Psi(x)$ satisfies the following global generalized
error bound property:
\begin{equation}
 \|x-\bar{x}\|_W \leq \left (\kappa_1+\kappa_2 \|x-\bar{x}\|_W^2 \right ) \|\nabla^+F(x) \|_W \quad \forall  x \in  \text{dom} \; \Psi,
\end{equation}
where $\kappa_1= (\tilde{\kappa}_1+ 2 \kappa_1'^2 \tilde{\kappa}_2) (2 L_{\Psi}+1)$  and $\kappa_2 = 2 \tilde{\kappa_2}  (2 L_{\Psi}+1)(2 L_{\Psi}^2+1)$.
\end{theorem}

\begin{proof}
From the previous discussion, it  remains to show that we can find an appropriate upper bound for $\|\nabla^+ \tilde{F}(z)\|_{\tilde{W}}$.
Given a point $z=[x^T \; \zeta]^T$, it can be observed that the gradient of $\tilde{F}(z)$ is:
\begin{equation*}
 \nabla \tilde{F}(z)=\begin{bmatrix} P^T \nabla \tilde{f}(Px) +c  \\ 1 \end{bmatrix} =\begin{bmatrix} \nabla f(x)  \\ 1 \end{bmatrix}.
\end{equation*}
Now, denote $z^+=\Pi_Z^{\tilde{W}} \left (z-\tilde{W}^{-1}\nabla \tilde{F}(z) \right )$. Following the definitions of the projection operator and of $\nabla^+ \tilde{F}$, note that $z^+$ is expressed as:
\begin{align*}
 z^+ =\arg  &  \min_{y \in \rset^n , \zeta' \in \rset}  \frac{1}{2} \left \|  \begin{matrix}  y- \left (x-W^{-1} \nabla f(x) \right )  \\ \zeta'- (\zeta -1) \end{matrix} \right \|_{\tilde{W}}^2\\
  \text{s.t.} & \quad \;  \Psi(y)\leq \zeta', \; \zeta'\leq \bar{\Psi}.
\end{align*}
Furthermore, from the definition of  $\|\cdot\|_{\tilde{W}}$, note that we can also express $z^+$ as:
\begin{align*}
z^+ =\arg & \quad \min_{y \in \rset^n , \zeta' \in \rset} \langle \nabla f(x), y-x \rangle+ \frac{1}{2} \|y-x\|_W^2+\frac{1}{2} (\zeta'-\zeta +1)^2  \\
\text{s.t.} & \quad \Psi(y) \leq \zeta', \; \zeta' \leq \bar{\Psi}. \nonumber
\end{align*}
Also, given the structure of $z$, consider that $z^+=[\tilde{T}_{[N]}(x)^T \; \zeta'' ]^T$. Now, by a simple change of variable, we can define a pair $(\tilde{T}_{[N]}(x),\tilde{\zeta})$
as follows:
\begin{align}
(\tilde{T}_{[N]}(x),\tilde{\zeta}) =\arg & \quad \min_{y \in \rset^n , \zeta' \in \rset} \langle \nabla f(x), y-x \rangle+ \frac{1}{2} \|y-x\|_W^2+\frac{1}{2} (\zeta'+1)^2 \label{problem_z} \\
\text{s.t.} & \quad \Psi(y) -\zeta \leq \zeta', \; \zeta' \leq \bar{\Psi}-\zeta. \nonumber
\end{align}
Note that $\tilde{\zeta}=\zeta''-\zeta$ and that we can express   $z^+=[\tilde{T}_{[N]}(x)^T \;\; \tilde{\zeta}+\zeta ]^T$ and:
\begin{equation*}
 \| \nabla^+ \tilde{F}(z)\|_{\tilde{W}}=\left \|\begin{matrix} x -\tilde{T}_{[N]}(x) \\ - \tilde{\zeta} \end{matrix} \right \|_{\tilde{W}}.
\end{equation*}

\noindent From \eqref{T_prox} and \eqref{prox_mapping}, we can write $\nabla^+ F(x)=x-T_{[N]}(x)$ and recall that $ T_{[N]}(x)$ can be expressed as:
\begin{equation*}
 T_{[N]}(x)=\arg \min_{y \in \rset^n} \langle \nabla f(x),y-x \rangle +\frac{1}{2} \|y-x\|_W^2 + \Psi(y)-\Psi(x).
\end{equation*}
Thus, we can consider that $ T_{[N]}(x)$ belongs to a pair  $(T_{[N]}(x),\hat{\zeta})$ which is the optimal solution of the following problem:
\begin{align}
(T_{[N]}(x),\hat{\zeta})  = \arg  &  \quad \min_{y \in \rset^n, \zeta' \in \rset } \langle \nabla f(x),y-x \rangle +\frac{1}{2} \|y-x\|_W^2 + \zeta'. \label{problem_pair} \\
&  \quad \text{s.t.:}  \quad \Psi(y) -\Psi(x) \leq \zeta'.  \nonumber
\end{align}
Following the same reasoning as in problem \eqref{problem2_psi}, note that $\hat{\zeta}=\Psi(T_{[N]}(x))-\Psi(x)$. Through  Fermat's rule \cite{RocWet:98} and problem \eqref{problem_pair}, we establish
that $(T_{[N]}(x),\hat{\zeta}) $ can also be expressed as:
\begin{align}
(T_{[N]}(x),\hat{\zeta})  = \arg  &  \quad \min_{y \in \rset^n, \zeta' } \langle \nabla f(x)+W(T_{[N]}(x)-x), y-x\rangle +\zeta' \\
&  \quad \text{s.t.:}  \quad \Psi(y) -\Psi(x) \leq \zeta'. \nonumber
\end{align}
Therefore, since $(T_{[N]}(x),\hat{\zeta})$ is optimal for the problem above, we establish the following inequality:
\begin{align}
 & \langle \nabla f(x)+W(T_{[N]}(x)-x),T_{[N]}(x) -x\rangle +\hat{\zeta} \nonumber  \\
 & \qquad \leq  \langle \nabla f(x)+W(T_{[N]}(x)-x), \tilde{T}_{[N]}(x)-x\rangle +\tilde{\zeta}. \label{add_1}
\end{align}
Furthermore, since the pair $(\tilde{T}_{[N]}(x), \; \tilde{\zeta})$ is optimal for problem \eqref{problem_z}, we can derive a second inequality:
\begin{align}
&\langle \nabla f(x), \tilde{T}_{[N]}(x)\!-\! x \rangle+ \frac{1}{2} \|\tilde{T}_{[N]}(x)-x\|_W^2+\frac{1}{2} (\tilde{\zeta}+1)^2  \label {add_2} \\
&\leq \langle \nabla f(x), T_{[N]}(x)\!-\! x \rangle+ \frac{1}{2} \|{T}_{[N]}(x)-x\|_W^2+\frac{1}{2} (\hat{\zeta}+1)^2. \nonumber
\end{align}
By adding up \eqref{add_1} and \eqref{add_2} we get the following relation:
\begin{align*}
 & \|{T}_{[N]}(x)-x\|_W^2 +\frac{1}{2} \|\tilde{T}_{[N]}(x)-x\|_W^2+\frac{1}{2} (\tilde{\zeta}+1)^2 +\hat{\zeta} \\
& \leq \frac{1}{2} \|{T}_{[N]}(x)-x\|_W^2 + \langle W(T_{[N]}(x)-x), \tilde{T}_{[N]}(x)-x\rangle +\frac{1}{2} (\hat{\zeta}+1)^2  +\tilde{\zeta}.
\end{align*}
If we further simplify this inequality we obtain:
\begin{align*}
\frac{1}{2} \|{T}_{[N]}(x)-x\|_W^2+ \frac{1}{2} \|\tilde{T}_{[N]}(x)-x\|_W^2 -\langle W(T_{[N]}(x)-x), \tilde{T}_{[N]}(x)-x\rangle+\frac{1}{2} \tilde{\zeta}^2 \leq \frac{1}{2} \hat{\zeta}^2.
\end{align*}
 Combining the first three terms in the left hand side under the  norm and if we multiply both sides by  $2$, the inequality becomes:
\begin{equation*}
\left \| (T_{[N]}(x)-x)-(\tilde{T}_{[N]}(x)-x)\right \|_W^2 + \tilde{\zeta}^2 \leq \hat{\zeta}^2.
\end{equation*}
From this, we derive the following two inequalities:
\begin{equation*}
\tilde{\zeta}^2 \leq \hat{\zeta}^2 \; \text{and} \;  \left \| (T_{[N]}(x)-x)-(\tilde{T}_{[N]}(x)-x)\right \|_W^2 \leq \hat{\zeta}^2.
\end{equation*}
If we take square root in both of these inequalities, and by applying the triangle inequality to the second, we obtain:
\begin{equation}
|\tilde{\zeta}| \leq |\hat{\zeta}| \; \text{and} \; \left \|\tilde{T}_{[N]}(x)-x\right \|_W-\| T_{[N]}(x)-x \|_W \leq |\hat{\zeta}|. \label{two_ineqs}
\end{equation}
Recall that $\hat{\zeta}=\Psi(T_{[N]}(x))-\Psi(x)$, and through the
Lipschitz continuity of $\Psi$,  we have from the first inequality
of \eqref{two_ineqs} that:
\begin{equation*}
 |\tilde{\zeta}| \leq |\hat{\zeta}| = |\Psi(T_{[N]}(x))-\Psi(x)| \leq L_{\Psi} \|T_{[N]}(x)-x\|_W.
\end{equation*}
Furthermore, from the second inequality of \eqref{two_ineqs} we obtain:
\begin{equation*}
\left \|\tilde{T}_{[N]}(x)-x\right \|_W \leq (L_{\Psi}+1) \|T_{[N]}(x)-x\|_W.
\end{equation*}
From these, we arrive at the following upper bound on $\|\nabla^+ \tilde{F}(z)\|$:
\begin{align}
 \|\nabla^+ \tilde{F}(z)\|&=\left \|\begin{matrix} x -\tilde{T}_{[N]}(x) \\ - \tilde{\zeta} \end{matrix} \right \|_{\tilde{W}} \leq \left \|\tilde{T}_{[N]}(x)-x\right \|_W  +|\tilde{\zeta}| \label{nabla_F_bound} \\
& \leq (2 L_{\Psi}+1) \|T_{[N]}(x)-x\|_W=(2 L_{\Psi}+1) \|\nabla^+ F(x) \|.\nonumber
\end{align}
Finally, from \eqref{bound_in_z}, \eqref{bound_squared} and
\eqref{nabla_F_bound} we  obtain the following error bound property
for problem \eqref{problem1_psi}:
\begin{equation*}
 \|x-\bar{x}\|_W \leq \left (\kappa_1 + \kappa_2 \| x - \bar{x} \|^2   \right )  \|\nabla^+ F(x)\|,
\end{equation*}
where $\kappa_1= (\tilde{\kappa}_1+ 2 \kappa_1'^2 \tilde{\kappa}_2) (2 L_{\Psi}+1)$
and $\kappa_2 = 2 \tilde{\kappa_2}  (2 L_{\Psi}+1)(2 L_{\Psi}^2+1)$.
\end{proof}


\subsection{Case 4:  dual formulation}
Consider now the following linearly constrained convex primal problem:
\begin{align}
\label{probl_primal}
\min_{u \in \rset^m} \{g(u): \; Au \leq b \}.
\end{align}
where $A \in \rset^{n \times m}$.  In many applications  however,
its dual formulation is used since the dual structure of the problem
is easier, see e.g. applications such as  network utility
maximization \cite{RamNed:09} or network control  \cite{NecSuy:09}.
Now,  for proving the generalized error bound property, we require
that $g$ satisfies the following assumption:
\begin{assumption}\label{ass_lip_strong}
We consider that $g$ is strongly convex (with constant $\sigma_g$)
and has Lipschitz continuous gradient (with constant $L_g$) w.r.t.
the Euclidean norm and there exists $\tilde u$ such that $A \tilde u
< b $.
\end{assumption}

\noindent Denoting by $g^*$ the convex conjugate of the  function
$g$, then from previous assumption it follows that $g^*$ is strongly
convex with constant $\frac{1}{L_g}$ and has Lipschitz gradient with
constant $\frac{1}{\sigma_g}$ (see e.g. \cite{RocWet:98}). Moreover,
from  the condition $A \tilde u < b $ it follows using Gauvin's
theorem that the set of optimal Lagrange multipliers is compact. In
conclusion, the previous primal problem is equivalent to the
following dual problem:
\begin{align}
\max_{x \in \rset^n} -g^*(-A^T x) - \langle x, b \rangle - \Psi(x),
\label{conv_conjug_EB}
\end{align}
where $\Psi(x)=\bo{I}_{\rset^n_+}(x)$ is the set indicator  function
for the nonnegative orthant $\rset^n_+$.   From Section
\ref{sec_case2}, for $P=-A^T$, it follows that the dual problem
\eqref{conv_conjug_EB} satisfies our generalized error bound
property  defined in \eqref{EEBF} (see
Definition~\ref{error_bound}).


\section{Convergence analysis under sparsity conditions}
\label{sec_comparison} In this section we analyze the distributed
implementation and the complexity of algorithm (\textbf{P-RCD})
w.r.t. the sparsity measure and compare it with other complexity
estimates from literature.


\subsection{Parallel and distributed implementation}
Nowadays, many big data applications which appear in the context of
networks  can be posed as problems of the form \eqref{gen_form}. Due
to the large dimension and the separable structure of these
problems, distributed optimization methods have become an
appropriate tool for solving such problems. From the iteration of
our algorithm  (\textbf{P-RCD}) it follows that we can efficiently
perform  parallel and/or distributed  computations. E.g., in the
case $\tau = N =\bar N$, we consider that each computer $i$ owns the
(block) coordinate $x_i$ and the function $f_i$ (provided that it
depends on $x_i$) and store them locally. Then, our iteration is
defined as follows:
\begin{align*}
x^{k+1}_i & =\arg \min_{y_i \in \rset^{n_i}} \langle \nabla_{i}
f(x^k), y_i-x^k_i \rangle +\frac{1}{2} \|y_i-x^k_i\|_{W_{ii}
I_{n_i}}^2 +\Psi_{i} (y_i) \quad \forall i \in [N],
\end{align*}
where the diagonal block components of the matrix $W=\text{diag}(W_{ii}; i \in [N])$ have the expression:
\begin{equation*}
 W_{ii} = \sum_{j \in \bar{\neigh}_i}L_{\neigh_j} I_{n_i} \quad \forall  i \in [N].
\end{equation*}
 Clearly, for updating $x^{k+1}_i$ we need to compute distributively $\nabla_{i} f(x^k)$  and
 $W_{ii}$. However,  $\nabla_{i} f(x)$ can be computed in a distributed fashion since
\[ \nabla_{i} f(x) = \sum_{j \in \bar{\neigh}_i}  \nabla_i f_j(x_{\neigh_j}), \]
i.e. node $i$ needs to collect the partial gradient $\nabla_i f_j$
from all the  functions  $f_j$ which depend on the variable $x_i$
(see also \cite{NecClip:13a} for more details on distributed
implementation of such an algorithm in the context of network
control). We can argue in a similar fashion for computing $W_{ii}$.
{\bl Also, for the case where $\tau \leq N$, we can employ other
distributed implementations for the algorithm such as the reduce-all
approach presented in \cite{RicTak:13}: if we consider that we have
a machine with $\tau$ available cores, then we can distribute the
information regarding the functions  and block-coordinates per
cores, i.e. each core will retain information regarding a multiple
number of coordinates and functions $f_j$. Then, we will require an
all-reduced strategy  for computing the $\nabla_i f(x)$.}

\noindent Further, through the norm $\|\cdot\|_W$,  which is inherent in $R_W(x^0)$,
convergence rates from Theorems \ref{theorem_sublinear} and
\ref{lin_converg} depend also  on the sparsity induced by the graph
via  the sets $\neigh_j$ and $\bar{\neigh}_i$.  As it can be observed,
the size of the diagonal elements  $W_{ii}$ depends on the values
of the Lipschitz constants $L_{\neigh_j}$,
with $j \in \bar{\neigh}_i$. Clearly these  constants $L_{\neigh_j}$
are influenced directly by the number $|\neigh_j|$ of variables that
a function $f_j$ depends on.  Moreover, $W_{ii}$ depends on the
number $| \bar{\neigh_i}|$  of individual functions $f_j$ in which
block component $x_i$ is found as an argument.  For example, let us consider the dual formulation \eqref{conv_conjug} of the primal problem  \eqref{prob_sum}.  In
this case  we have $L_{\neigh_j} =
\frac{\|A_{\neigh_j}\|^2}{\sigma_j}$. Given that the matrix block
$A_{\neigh_j}$ is composed of  blocks $A_{lj}$, with $ l \in
\neigh_j$, and from the definition of $\omega$ we have the following
inequality:
\begin{align*}
 L_{\neigh_j} & = \frac{\|A_{\neigh_j}\|^2}{\sigma_j} \leq \sum_{l \in \neigh_j}
 \frac{\|A_{lj}\|^2}{\sigma_j} \leq \omega \max_{ l \in \neigh_j}
 \frac{\|A_{lj}\|^2}{\sigma_j} \quad \forall j.
\end{align*}
Furthermore, from this inequality and definition of $\bar{\omega}$,
the diagonal terms of the matrix $W$
 can be expressed as:
\begin{equation*}
W_{ii} =  \sum_{j \in \bar{\neigh}_i} L_{\neigh_j} \leq
\bar{\omega}\max_{j \in \bar{\neigh_i}} L_{\neigh_j} \leq \omega
\bar{\omega}\max_{l \in \neigh_j, j \in \bar{\neigh_i}}
\frac{\|A_{lj}\|^2}{\sigma_j} \quad \forall i.
\end{equation*}
Thus, from the previous inequalities we derive the following upper
bound:
\begin{equation*}
 \left (R_W(x^0) \right )^2  \leq \omega \bar{\omega}
 \left( \max_{l \in \neigh_j, j \in \bar{\neigh_i}} \frac{\|A_{lj}\|^2}{\sigma_j} \right)
 \left (R_{I_n}(x^0) \right )^2.
\end{equation*}
In conclusion, our measure of separability $(\omega, \bar{\omega})$
for the original problem \eqref{gen_form}  appears implicitly in the
estimates on the convergence rate for our algorithm
\textbf{(P-RCD)}. On the other hand,  the estimate on the
convergence rate in \cite{RicTak:12a} depends on the maximum number
of connections which a subsystem has, i.e. only on $\omega$. This
shows that our approach is more general, more flexible and thus
potentially less conservative, as we will  also see in the next
section.


\subsection{Comparison with other approaches}
In this section we  compare our convergence rates with those from
other existing methods under sparsity conditions. Recall that under
Assumption \ref{lip_fi} a function $f$ satisfies the  lemma given in \eqref{desc_lemma}:
\begin{align}
\label{lemma_descent_Necoara}
 f(y) \leq f(x) + \langle \nabla f(x),y-x \rangle+\frac{1}{2} \|y-x\|_{W}^2,
\end{align}
property which we have employed throughout the paper.
The essential element in this relation is the sparsity induced by
the sets $\neigh_j$ and $\bar{\neigh_i}$, which are reflected in the
matrix $W$.  Nesterov  proves in  \cite{Nes:12}, under the
coordinate-wise Lipschitz  continuous gradient assumption
\eqref{lipschitz_grad}  and without any separability property, the following  descent lemma for functions $f$:
\begin{align}
\label{lemma_descent_Nesterov}
 f(y) \leq f(x) + \langle \nabla f(x),y-x \rangle+\frac{N}{2} \|y-x\|_{W'}^2,
\end{align}
where the matrix $W'=\text{diag} \left (L_i  I_{n_i}; i \in [N]
\right)$, with  $L_i$ being the  Lipschitz constants such that $f$
satisfies \eqref{lipschitz_grad}. In   \cite{RicTak:12a},  under an
additional separability assumption on the function $f$, Nesterov's descent lemma \eqref{lemma_descent_Nesterov} was
generalized as follows:
\begin{equation}
 f(y) \leq f(x) + \langle \nabla f(x),y-x \rangle+\frac{\omega}{2} \|y-x\|_{W'}^2, \label{desc_lemma_richt}
\end{equation}
where $\omega$ is defined in Section \ref{sec_probdef}. In order to
be able to compare the convergence rates of our method with existing
convergence results we assume below that $L_{\neigh_j}$ and $L_i$
are of the same magnitude.

\noindent  \textbf{Sublinear convergence case:} Recall that the
sublinear convergence rate of our algorithm \textbf{(P-RCD)}, that
holds  under Assumption \ref{lip_fi}, is (see Theorem
\ref{theorem_sublinear}):
\begin{equation}
 \average [F(x^k)]-F^* \leq   \frac{2 N c }{\tau k+ N}
 \quad \forall k \geq 0, \label{pcdm_converg_2}
\end{equation}
{\bl where $c= \max \left \{ 1/2 (R_W(x^0))^2, F(x^0)-F^* \right
\}$. For $\tau=1$ we obtain a similar convergence rate to that of
the random coordinate descent method in \cite{Nes:12}, i.e. of order
${\cal O}(Nc/k)$,
while for $\tau=N$ we get  a similar convergence
rate to that of the full composite gradient method of \cite{Nes:13}. However,
the  distances are measured in different norms in these papers.
{\bl For example, when $\tau=N$ the comparison of convergence rate
in our paper and \cite{Nes:13} is reduced to comparing the
quantities $L_f R(x^0)^2$ of \cite{Nes:13} with our amount
$R_W(x^0)^2$, where $L_f$ is the Lipschitz constant of the smooth
component of the objective function, i.e. of $f$, while $R(x^0)$ is
defined in a similar fashion as our $R_W(x^0)$ but in the Euclidean
norm, instead of the norm $\|\cdot \|_W$. Let us consider the two
extreme cases. First, consider the smooth component of the objective
function:
$$f(x)=\sum_{j=1}^{\bar N} f_j(x_j),$$ i.e. smooth part is  fully separable.
Recall that we assume that each individual function $f_j$ is
Lipschitz continuous with a constant $L_{\neigh_j}$, as stated in
Assumption \ref{lip_fi}. In this case, it can be easily proven that
the Lipschitz constant of $f$ is $L_f= \max_{j  \in[\bar N]}
L_{\neigh_j}$. Thus considering the definition of the matrix $W$ in
Lemma \ref{lema_desc} we have that:
$$ L_f \|x^0-x^*\|^2 \geq \|x^0-x^*\|_W^2, $$ }
\noindent {\bl i.e. $L_f R(x^0)^2 \leq R_W(x^0)^2$ and our convergence rate is
usually  better. On the other hand, if we have $f$ defined as
follows:
$$f(x)=\sum_{j=1}^{\bar N} f_j(x),$$
then  it can be easily proven that  $L_f=\sum_{j \in \bar N}
L_{\neigh_j}$ and the quantities $L_f R(x^0)^2$ and $R_W(x^0)^2$
would be  the same. Thus, we get better rates of convergence when
$\bar{\omega} < \bar{N}$.}  {\bl Finally, we notice that our results are also  similar with those of \cite{RicTak:12a}, but are obtained under
further knowledge regarding the objective function and with a
modified analysis.  In particular,  in \cite{RicTak:12a} two algorithms  are proposed:
algorithms (PCDM1) and (PCDM2) which explicitly enforces
monotonicity. However, in practical large scale applications
algorithm (PCDM2) cannot be implemented due to the very large cost
per iteration, as  the authors also state in their paper.}  Thus,
using a similar reasoning as in
\cite{LuXia:13},  we can argue, based on our analysis, that the
expected value type of convergence rate given in \eqref{pcdm_conv_1}
is better than the one in \cite{RicTak:12a} under certain
separability properties as described below. First, the convergence
rate of the algorithm in the sublinear case, apart from essentially
being of order ${\cal O}\left( \frac{1}{k} \right )$, depends on the
quantity $c$, i.e. implicitly on the stepsizes involved when
computing the next iterate $x^{k+1}$, see \eqref{iter_pcdm2}. Thus,
in essence, the requirement is to find the smallest values for the
diagonal elements of matrix $W$ such that Lemma \ref{lema_desc} is
still valid. To this purpose, we consider the smooth component in
\eqref{sol_norm} in the form $f(x)=\frac{1}{2} \|Ax-b\|^2$.  For
this problem, consider basic block coordinate, i.e. $n_i=1$. Under
these considerations, we observe from the table below that our
stepsizes are better than those in \cite{RicTak:12a} as $\tau$
increases and $\bar{\omega} \ll \omega$. Thus, for certain cases,
our analysis can show improvement over the stepsizes  in
\cite{RicTak:12a} under the same sampling.
\begin{table}[!h]
\begin{center}
\setlength{\extrarowheight}{1.5pt}
 \begin{tabular}{|l|c|}
\hline Paper & $W_{ii}$ \\
\hline This paper & $\sum_{j: i \in \neigh_j} \sum_{t=1}^n A_{jt}^2$ \\
\hline \cite{RicTak:12a} & PCDM1: $\sum\limits_{j=1}^{\bar{N}} \! \min(\omega,\tau)
A_{ji}^2$ \;\; or \;\; PCDM2: $\sum\limits_{j=1}^{\bar{N}} \!\! \left (
1+\frac{(\omega-1)
(\tau-1)}{\max \{1,n-1\} } \right ) A_{ji}^2$ \\
\hline
\end{tabular}
\end{center}
\end{table}
}

\noindent {\bl Finally,  we   proceed  to compare the convergence of
our algorithm with the algorithms in \cite{RicTak:12a}.  Since in practical large scale applications algorithm (PCDM2) in \cite{RicTak:12a} cannot be implemented due to the very large cost per iteration, as  the authors also state in their paper,
in the sequel we consider algorithm (PCDM1) in \cite{RicTak:12a} under the $\tau$-uniform
sampling, which is similar with our sampling strategy, and for which
the authors of \cite{RicTak:12a} were able to derive rate of
convergence. In this case, both algorithms \textbf{(P-RCD)} and
(PCDM1) have similar costs per iteration.} In this setting, (PCDM1)
has the following sublinear convergence:
\begin{equation}
 \average [F(x^k)]-F^* \leq   \frac{2 N c' }{\tau k+ 2N c'(F(x^0) - F^*)^{-1}}
 \quad \forall k \geq 0,  \label{converg_Richt}
\end{equation}
where we define $\beta = \min(\omega,\tau)$ and
\begin{align}
& c'= \max \left \{ \beta  (R_{W'}(x^0))^2, F(x^0)-F^* \right \} \nonumber\\
& R_{W'}(x^0) = \max_{x: \; F(x)\leq F(x^0)} \min_{x^* \in X^*}
\|x-x^*\|_{W'}. \label{R_richt}
\end{align}
Consider that in both algorithms  we have that $F(x^0)-F^*$ is the
smallest term in the two maximums, i.e. $1/2(R_{W}(x^0))^2 \geq
F(x^0)-F^*$ and  $\beta (R_{W'}(x^0))^2 \geq F(x^0)-F^*$. Let us
make a comparison between  the two convergence rates,
\eqref{pcdm_converg_2}  and \eqref{converg_Richt} and note that this
comparison comes down to the comparison between the norms
$\|\cdot\|_W^2$ and the quantity $\beta \| \cdot \|_{W'}^2$.  From
the definitions on the norms we can express:
\begin{align*}
\|x\|_W^2&=\sum_{i=1}^N \left (\sum_{j \in \bar{\neigh_i}}
L_{\neigh_j} \right ) \|x_i\|^2 \\
\beta \|x\|_{W'}^2&= \sum_{i=1}^N \min(\omega,\tau)  L_i \|x_i\|^2.
\end{align*}

\noindent In  conclusion, the sublinear convergence rate of
algorithm  \textbf{(P-RCD)} is {\bl improved under our assumptions,
for similarly sized Lipschitz  constants and  for $\bar{\omega}  \ll
\omega \leq \tau$, i.e. for problems where we have at least a
function $f_j$ that depends on a large number of variables (i.e.
$\omega$ relatively large) but each variable does not appear in many
functions $f_j$ (i.e. $\bar{\omega}$ relatively small)}. Note that
this scenario was also considered at the beginning of the paper
since coordinate gradient descent type methods for solving problem
\eqref{gen_form} make sense only in the case  when $ \bar{\omega}$
is small, otherwise incremental type methods  \cite{WanBer:13}
should be considered for solving \eqref{gen_form}.

\vspace*{0.2cm}

\noindent \textbf{Linear convergence case:}   The authors of
\cite{Nes:12,Nec:13,NecNes:11,NecClip:13a,NecPat:12,RicTak:12a,LuXia:13}
also provide  linear convergence rate for their algorithms. A
straightforward comparison between the convergence rates in this
paper and of those in the papers mentioned above  cannot be done,
due to the fact the linear convergence in all these papers is proved
under the more conservative assumption of strong convexity, while
the convergence rate  of our algorithm \textbf{(P-RCD)} is obtained
under the more relaxed assumption of generalized error bound
property \eqref{EEBF} given in Definition \ref{error_bound}.
However, we can also consider $f$ to be strongly convex with a
constant $\sigma_{W}$ w.r.t. the norm $\|\cdot\|_W$. From Section
\ref{case1_strong} it follows that strongly convex functions are
included in our class of generalized error bound functions
\eqref{EEBF} with $\kappa_1 = \frac{2}{\sigma_W }$ and $\kappa_2=0$.
In this case we can easily prove that we  have the following linear
convergence of \textbf{(P-RCD)}:
\begin{equation}
\average \left [F(x^{k+1})  -F^* \right ] \leq
(1-\gamma_{sc}^{eb})^k \left (F(x^0) -F^* \right ),
\label{linear_second}
\end{equation}
where $\gamma^{eb}_{sc}=\frac{\tau \sigma_W}{N}$. Note that from
\eqref{desc_lemma} it follows that the Lipschitz constant of the
gradient is equal to $1$ and then combining \eqref{desc_lemma} with
\eqref{strongconvexity} we get $\sigma_W \leq 1$  w.r.t. the norm
$\|\cdot\|_W$. In this case, we notice that, given the choice of
$\tau$, we obtain different linear convergence results of order
${\mathcal O}(\theta^k)$. E.g., for $\tau=1$ we obtain a similar
linear convergence rate to that of the random coordinate descent
method in \cite{Nec:13,Nes:12,LuXia:13,RicTak:12}, i.e.
$\gamma^{eb}_{sc}={\mathcal O}(\sigma_W/N)$, while for $\tau=N$ we
get  a similar convergence rate to that of the full composite
gradient method of \cite{Nes:13}, i.e $\gamma^{eb}_{sc}={\mathcal
O}(\sigma_W)$. Finally, if we consider $f$ to be strongly convex in
the norm $\omega \|\cdot\|_{W'}$ with a constant $\sigma_{W'}$ and
under the same sampling strategy as considered here, then for
example the algorithm (PCDM1) in \cite{RicTak:12a}  has a
convergence rate:
\begin{equation}
\average \left [F(x^{k+1})  -F^* \right ]\leq (1-\gamma_{sc})^k \left (F(x^0) -F^* \right), \label{linear_ric}
\end{equation}
where $\gamma_{sc}=\frac{\tau \sigma_{W'}}{N + \omega \tau}$. Thus,
the comparison of convergence rates in this case  reduces to the
comparison of $\gamma_{sc}^{eb}$ and $\gamma_{sc}$.  {\bl Then, if
$\bar{\omega}$ is sufficiently small and $\omega$ is sufficiently
large,  we get that $\gamma_{sc}^{eb} \geq \gamma_{sc}$ and the
linear  convergence rate \textbf{(P-RCD)} is an improvement over
that of \cite{RicTak:12a}}. Thus,  we found again that under some
degree of sparsity (i.e. $ \bar{\omega} \ll \omega$)  our results
are superior to those in \cite{RicTak:12a}.

\noindent Our convergence results are also more general than the
ones in \cite{Nes:12,Nec:13,NecPat:12,RicTak:12a,LuXia:13}, in the
sense that we can show linear convergence of algorithm
\textbf{(P-RCD)} for larger classes of problems. For example, up to
our  knowledge the best global convergence rate results known  for
gradient type methods   for  solving optimization problems of the
form  dual formulation of a linearly constrained convex problem
\eqref{conv_conjug} or general constrained lasso \eqref{sol_norm}
were of the sublinear form ${\cal O}\left( \frac{1}{k^2} \right)$
\cite{Nes:13,NecNed:13}.  In this paper we show for the first time
\textit{global} linear convergence  rate for random coordinate
gradient descent methods for solving this type of problems
\eqref{sol_norm} or \eqref{conv_conjug}. Note that for the
particular case of least-square problems $ \min_x \| Ax - b\|^2$ the
authors in \cite{LevLew:08}, using also an error bound like
property, were able to show linear convergence for a random
coordinate gradient descent method. Our results can be viewed as a
generalization of the results from \cite{LevLew:08} to more general
optimization problems. Further, our approach allows us to analyze in
the same framework several methods: full gradient, serial coordinate
descent and any parallel coordinate descent method in between.


\section{Numerical simulations}
In this section we present  some preliminary numerical results on
solving constrained lasso problems in the form  \eqref{sol_norm}.
The individual constraint sets $X_i \subseteq \rset^{n_i}$ are box
constraints, i.e. $X_i=\left \{x_i : lb_i \leq x_i \leq ub_i \right
\}$. The regularization parameters $\lambda_i$ were chosen uniform
for all components, i.e. $\lambda_i = \lambda$ for all $i$. The
numerical experiments were done for two instances of the
regularization parameter $\lambda=1$ and $\lambda=10$. The numerical
tests were conducted on a machine with 2 Intel(R) Xeon(R) E5410 quad
core CPUs  @ 2.33GHz  and $8GB$ of RAM. The matrices $A \in \rset^{m
\times n}$ were randomly generated in Matlab and have a high degree
of sparsity (i.e. the measures of partial separability $\omega,
{\bar \omega} \ll n$).

\noindent In the first experiment we solve a single randomly
generated constrained lasso problem with matrix $A$ of dimension
$m={\bar N} = 0.99*10^6$ and $n=N = 10^6$. In this case the two
measure of separability  have the  values: $\omega=37$  and $\bar
\omega = 35$. The problem was solved on $\tau=1, 2, 4$ and $7$ cores
in parallel using MPI for $\lambda=10$.
  From  Fig. 8.1 and 8.2 we can observe that for each $\tau$ our algorithm needs almost the
  same number of coordinate updates $\frac{\tau k}{n}$ to solve the problem.
  On the other hand increasing the number of cores reduces substantially the
  number of iterations  $\frac{k}{n}$.
\begin{figure}[ht]
\centering
\begin{minipage}[b]{0.45\linewidth}
\includegraphics[width=\textwidth]{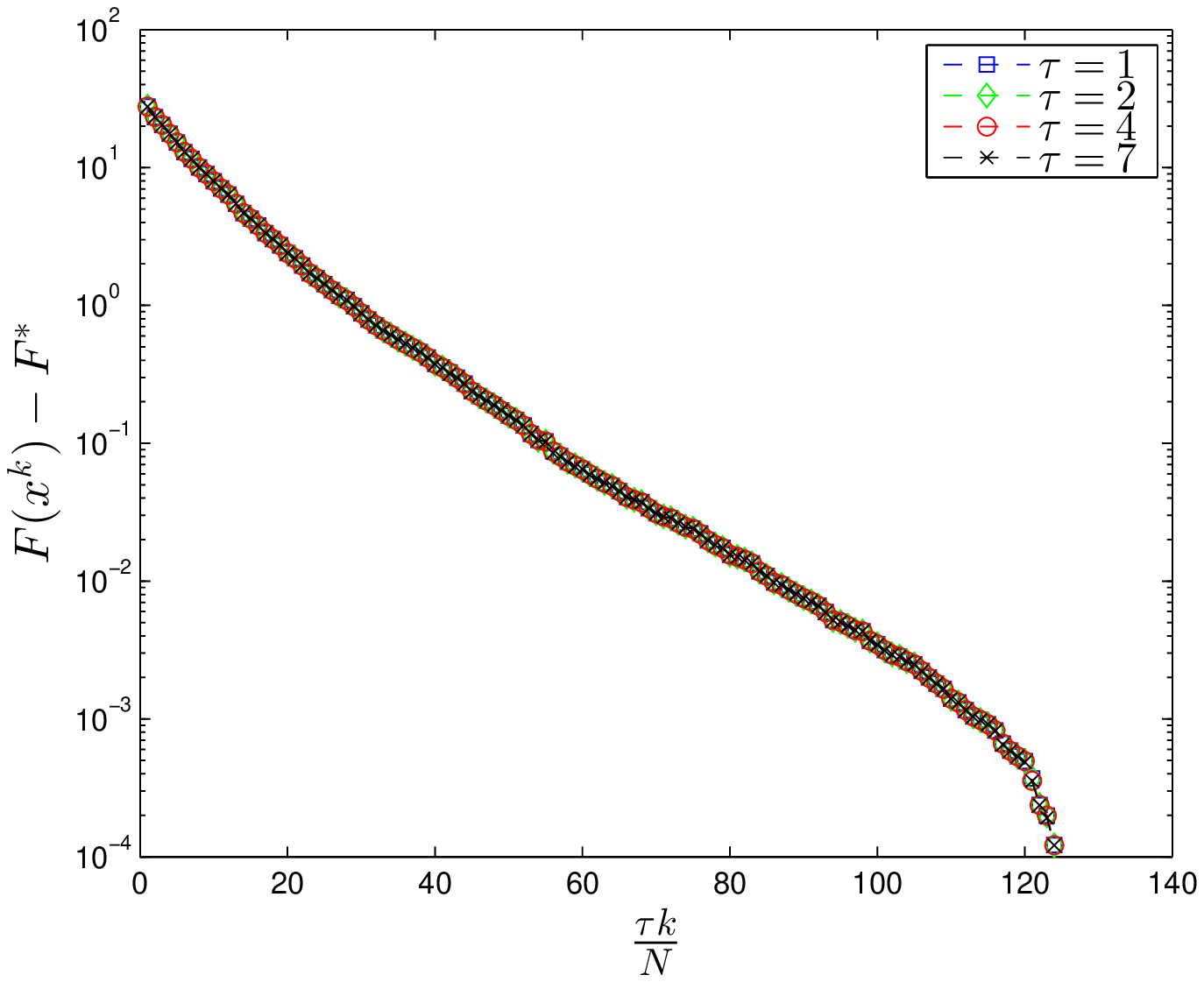}
\caption{Evolution of $F(x^k)-F^*$ along  coordinate updates normalized $\frac{\tau k}{n}$.}
\label{decrease:taukn}
\end{minipage}
\quad
\begin{minipage}[b]{0.45\linewidth}
\includegraphics[width=\textwidth]{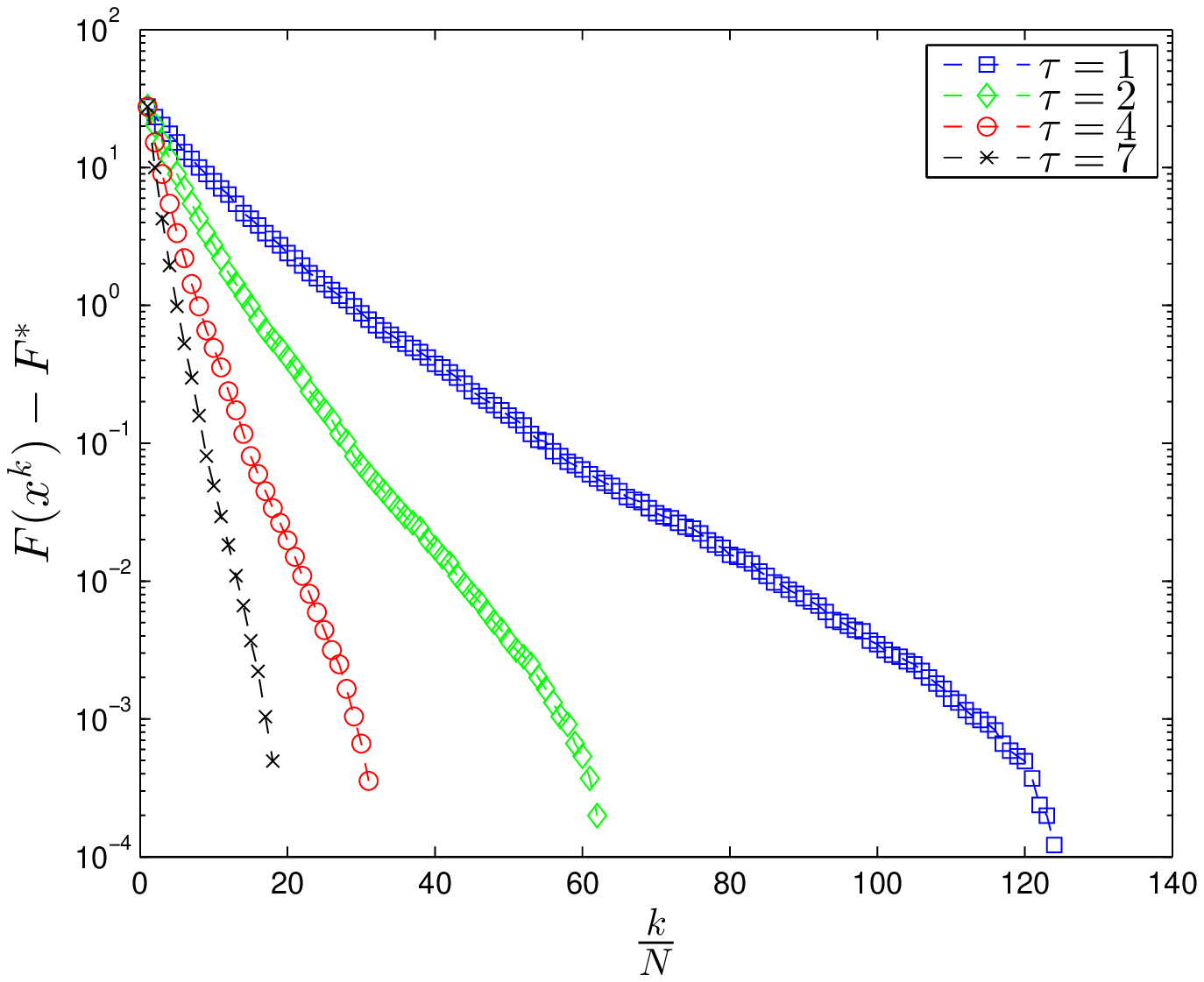}
\caption{ Evolution of $F(x^k)-F^*$ along  iterations normalized $\frac{k}{n}$. }
\label{decrease:kn}
\end{minipage}
\end{figure}
\begin{figure}[htb]
\begin{center}
 \includegraphics[scale=0.35]{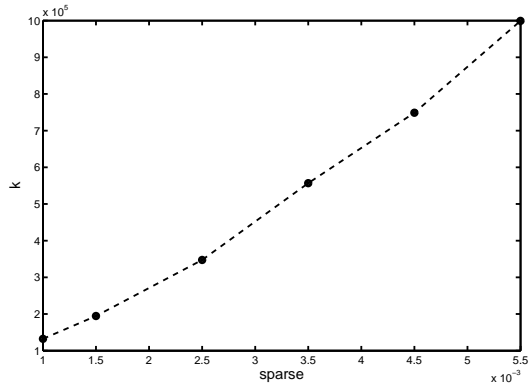}
\caption{ Average number of iterations for alg. \textbf{(P-RCD)} for
each  $10$ randomly generated problems with varying sparsity. }
\label{graph_10_4average}
\end{center}
\end{figure}

\noindent Then, in Fig. \ref{graph_10_4average} for each $10$
randomly generated problems with varying sparsity ranging from
$10^{-3}$ to $6*10^{-3}$  and  dimension $m={\bar N} = 0.9*10^4$ and
$n=N = 10^4$ we plot the average number of iterations. We consider
$\tau=50$  and $\lambda=10$.  We observe that the number of
iterations increases once the sparsity decreases.
\begin{table}[!htb]
\begin{center}
\setlength{\extrarowheight}{0.5pt}
\begin{footnotesize}
 \begin{tabular}{|c|c|r|r|c|c|>{\hspace{-5pt}}c<{\hspace{-5pt}}|>{\hspace{-5pt}}c<{\hspace{-5pt}}|>{\hspace{-5pt}}c<{\hspace{-5pt}}|}
\hline
  $\lambda$  & $n$ & $m$ & $sparse$ & $\bar{\omega}$ & $\omega$ & $ \tau k_{\text{(P-RCD)}} /n$ & $\tau k_{\text{\cite{RicTak:12a}}}/n$ & $F^*$ \\
 \hline \multirow{15}{*}{1}     & \multirow{5}{*}{$10^4$}  & $0.90 \times 10^ 4$ & $2 \times 10^{-3}$  & 35 & 38 & 177 & 279 & 2420.107  \\
        \cline{3-9}    &  & $0.98 \times 10^ 4$  & $3 \times 10^{-3}$  & 55 & 52 & 274 & 432 & 2379.622  \\
        \cline{3-9}    &  & $0.94 \times 10^ 4$  & $4 \times 10^{-3}$  & 64 & 63 & 418 & 605 & 1985.261  \\
         \cline{3-9}   &   & $1    \times 10^ 4$ & $4 \times 10^{-3}$  & 71 & 66 & 364 & 556 & 2422.455  \\
        \cline{3-9}    &   & $1.03 \times 10^ 4$ & $4 \times 10^{-3}$  & 68 & 69 & 397 & 635 & 2307.750  \\

\hhline{|~|=|=|=|=|=|=|=|=|} & \multirow{5}{*}{$10^5$} & $0.97 \times 10^ 5$  & $1.3 \times 10^{-4}$  & 31 & 32 & 111 & 193 & 27768.840  \\
         \cline{3-9}     &  & $0.91 \times 10^ 5$  & $1.5 \times 10^{-4}$  & 34 & 36 & 128 & 238 & 25918.885  \\
         \cline{3-9}     &  & $0.93 \times 10^ 5$  & $2 \times 10^{-4}$  & 43 & 41 & 167 & 285 & 25860.573  \\
         \cline{3-9}     &  & $1    \times 10^ 5$  & $2 \times 10^{-4}$  & 41 & 42 & 162 & 280 & 26894.849  \\
         \cline{3-9}     &  & $1.046 \times 10^ 5$  & $2 \times 10^{-4}$  & 42 & 42 & 161 & 270 & 28405.369  \\

\hhline{|~|=|=|=|=|=|=|=|=|} & \multirow{5}{*}{$10^6$}  & $0.98 \times 10^ 6$  & $1.5 \times 10^{-5}$  & 40 & 38 & 119 & 207 & 287300.02  \\
         \cline{3-9}         &              & $0.91 \times 10^ 6$  & $1.7 \times 10^{-5}$              & 34 & 36 & 144 & 235 & 251255.96  \\
         \cline{3-9}         &              & $0.99 \times 10^ 6$  & $2 \times 10^{-5}$                & 43 & 44 & 109 & 229 & 227031.21   \\
         \cline{3-9}         &              & $1    \times 10^ 6$  & $2 \times 10^{-5}$                & 46 & 43 & 101 & 187 & 273215.09   \\
         \cline{3-9}         &              & $1.046 \times 10^ 6$  & $2 \times 10^{-5}$               & 51 & 53 & 99 &  182 & 239189.71   \\

\hhline{|=|=|=|=|=|=|=|=|=|} \multirow{15}{*}{10}  & \multirow{5}{*}{$10^4$} &
$0.98 \times 10^ 4$ & $2 \times 10^{-3}$  & 39 & 42 & 24 & 38 & 4884.610  \\
        \cline{3-9}    &  & $0.96 \times 10^ 4$  & $3 \times 10^{-3}$  & 51 & 52 & 38 & 62 & 4762.226  \\
        \cline{3-9}    &  & $0.92 \times 10^ 4$  & $4 \times 10^{-3}$  & 64 & 70 & 52 & 85 & 4477.707  \\
        \cline{3-9}   &   & $1    \times 10^ 4$ & $4 \times 10^{-3}$  & 65 & 65 & 57 & 81 & 4909.406  \\
        \cline{3-9}    &   & $1.02 \times 10^ 4$ & $4 \times 10^{-3}$  & 68 & 68 & 51 & 83 & 4922.320  \\

\hhline{|~|=|=|=|=|=|=|=|=|} & \multirow{5}{*}{$10^5$}  & $0.92 \times 10^ 5$  & $1.3 \times 10^{-4}$  & 34 & 31 & 13 & 28 & 46066.411  \\
         \cline{3-9}     &  & $0.95 \times 10^ 5$  & $1.5 \times 10^{-4}$  & 32 & 35 & 16 & 33 & 47770.23  \\
         \cline{3-9}     &  & $0.91 \times 10^ 5$  & $2 \times 10^{-4}$  & 40 & 46 & 23 & 43 & 45520.275  \\
         \cline{3-9}     &  & $1    \times 10^ 5$  & $2 \times 10^{-4}$  & 42 & 43 & 23 & 43 & 49808.196  \\
         \cline{3-9}     &  & $1.09 \times 10^ 5$  & $2 \times 10^{-4}$  & 46 & 43 & 22 & 41 & 54370.699  \\

\hhline{|~|=|=|=|=|=|=|=|=|} & \multirow{5}{*}{$10^6$} & $0.9 \times 10^ 6$   & $1.5 \times 10^{-5}$  & 35 & 37 & 14 & 26 & 449548.04  \\
         \cline{3-9}         &                 & $0.91 \times 10^ 6$  & $1.7 \times 10^{-5}$          & 41 & 40 & 18 & 33 & 467529.31  \\
         \cline{3-9}         &                 & $0.97 \times 10^ 6$  & $2 \times 10^{-5}$            & 42 & 43 & 22 & 42 & 452739.23  \\
         \cline{3-9}         &                         & $1    \times 10^ 6$  & $2 \times 10^{-5}$    & 43 & 44 & 17 & 36 & 426963.31  \\
         \cline{3-9}         &                         & $1.1 \times 10^ 6$  & $2 \times 10^{-5}$     & 48 & 43 & 19 & 39 & 442936.02  \\
\hline 
 \end{tabular}
\end{footnotesize}
\caption {Comparison of algorithms \textbf{(P-RCD)} and (PCDM1) of
\cite{RicTak:12a}.} \label{tabel_sim}
\end{center}
\end{table}




\noindent In the second set of experiments, provided in  Table
\ref{tabel_sim},  the dimension of matrix $A$ ranges as follows:
$m={\bar N}$ from $0.9*10^4$ to $1.1*10^6$ and $n=N$ from $10^4$ to
$10^6$.  For the resulting problem  our objective function satisfies
the generalized error bound property \eqref{EEBF} given in
Definition \ref{error_bound} and in some cases it is even strongly
convex. This  series of numerical tests were undertaken in order to
compare the full number of iterations of the algorithm {\bl under
the original assumptions considered in \cite{RicTak:12a}  and the  ones
considered for the algorithm in this paper}.

\noindent In these simulations
the algorithms were implemented in a centralized manner, i.e.  there
is no inter-core transmission of data,  with the number of updates
per iteration of $\tau=100$ in each case.  {\bl In both cases the
algorithms were allowed to reach the same optimal value $F^*$ which
is presented in the last column  and was computed with the serial
($\tau=1$) random coordinate descent method.} The second and third
column of the table represent the dimensions of matrix $A$. The
fourth column represents the degree of sparsity which
 dictates that the total number of nonzero
elements in the matrix $A$ is less than or equal to $n \times m
\times sparse$. The fifth and sixth columns denote the degrees of
partial separability $\bar{\omega}$ and $\omega$, while the seventh
and eighth columns represent the total number of coordinate updates
normalized that the algorithms completed. {\bl  As it can be
observed from Table \ref{tabel_sim}, algorithm  \textbf{(P-RCD)}
outperforms (PCDM1) of \cite{RicTak:12a} even in the case where
$\bar{\omega}$ and $\omega$ are of similar size or equal}. Moreover,
note that between the problems where $m$ is slightly larger than
$n$, i.e. where the resulting objective function $F(x)$ is  strongly
convex, and the problems where $m$ is slightly smaller than $n$,
i.e. where $F(x)$ is not strongly convex but satisfies our
generalized error bound property \eqref{EEBF}, the number of
iterations of algorithm \textbf{(P-RCD)} is comparable. In
conclusion, given that the constrained lasso problems of the form
\eqref{sol_norm} satisfy the generalized error bound property
\eqref{EEBF}, the theoretical result that linear convergence of
algorithm \textbf{(P-RCD)} is attained under the generalized
error bound property is confirmed also in practice. \vspace{1cm}

\bibliographystyle{unsrt}
\bibliography{bibliografie2013}

\begin{thebibliography}{10}

\bibitem{NecClip:13a}
I.~Necoara and D.~Clipici.
\newblock Efficient parallel coordinate descent algorithm for convex
  optimization problems with separable constraints: application to distributed
  mpc.
\newblock {\em Journal of Process Control}, 23(3):243--253, 2013.

\bibitem{NecSuy:09}
I.~Necoara and J.A.K. Suykens.
\newblock An interior-point lagrangian decomposition method for separable
  convex optimization.
\newblock {\em Journal of Optimization Theory and Applications},
  143(3):567--588, 2009.

\bibitem{Bis:06}
C.M. Bishop.
\newblock {\em Pattern Recognition and Machine Learning}.
\newblock Springer-Verlag, 2006.

\bibitem{WitFra:06}
I.H. Witten, E.~Frank, and M.A. Hall.
\newblock {\em Data Mining: Practical Machine Learning Tools and Techniques}.
\newblock Elsevier, New York, 2011.

\bibitem{RicTak:13}
P.~Richtarik and M.~Takac.
\newblock Distributed coordinate descent method for learning with big data.
\newblock Technical report, Univ. Edinburgh, Oct. 2013.

\bibitem{BecTet:13}
A.~Beck and L.~Tetruashvili.
\newblock On the convergence of block coordinate descent type methods.
\newblock {\em SIAM Journal on Optimization}, 23(4):2037--2060, 2013.

\bibitem{HonWan:13}
M.~Hong, X.~Wang, M.~Razaviyayn, and Z-Q. Luo.
\newblock Iteration complexity analysis of block coordinate descent methods.
\newblock Technical report, University of Minnesota, USA, 2013.
\newblock http://arxiv.org/abs/1310.6957.

\bibitem{TseYun:09}
P.~Tseng and S.~Yun.
\newblock A coordinate gradient descent method for nonsmooth separable
  minimization.
\newblock {\em Mathematical Programming}, 117(1--2):387--423, 2009.

\bibitem{Tse:01}
P.~Tseng.
\newblock Convergence of a block coordinate descent method for
  nondifferentiable minimization.
\newblock {\em Journal of Optimization Theory and Applications},
  109(3):475--494, 2001.

\bibitem{Nes:12}
Y.~Nesterov.
\newblock Efficiency of coordinate descent methods on huge-scale optimization
  problems.
\newblock {\em SIAM Journal on Optimization}, 22(2):341--362, 2012.

\bibitem{Nec:13}
I.~Necoara.
\newblock Random coordinate descent algorithms for multi-agent convex
  optimization over networks.
\newblock {\em IEEE Trans. Automatic Control}, 58(8):2001--2012, 2013.

\bibitem{NecNes:11}
I.~Necoara, Y.~Nesterov, and F.~Glineur.
\newblock A random coordinate descent method on large optimization problems
  with linear constraints.
\newblock Technical report, University Politehnica Bucharest, July 2011.

\bibitem{NecPat:12}
I.~Necoara and A.~Patrascu.
\newblock A random coordinate descent algorithm for optimization problems with
  composite objective function and linear coupled constraints.
\newblock {\em Computational Optimization and Applications}, 57(2):307--337,
  2014.

\bibitem{Nes:13}
Y.~Nesterov.
\newblock Gradient methods for minimizing composite objective functions.
\newblock {\em Mathematical Programming}, 140(1):125--161, 2013.

\bibitem{RicTak:12}
P.~Richtarik and M.~Takac.
\newblock Iteration complexity of randomized block-coordinate descent methods
  for minimizing a composite function.
\newblock {\em Mathematical Programming}, 144:1--38, 2014.

\bibitem{LuXia:13}
Z.~Lu and L.~Xiao.
\newblock On the complexity analysis of randomized block-coordinate descent
  methods.
\newblock Technical report, 2013.
\newblock http://arxiv.org/abs/1305-4723.

\bibitem{LuoTse:93}
Z.Q. Luo and P.~Tseng.
\newblock Error bounds and convergence analysis of feasible descent methods: a
  general approach.
\newblock {\em Annals Operations Research}, 46(1):157--178, 1993.

\bibitem{RicTak:12a}
P.~Richtarik and M.~Takac.
\newblock Parallel coordinate descent methods for big data optimization.
\newblock Technical report, Univ. Edinburgh, Dec. 2012.

\bibitem{PenYan:13}
Z.~Peng, M.~Yan, and W.~Yin.
\newblock Parallel and distributed sparse optimization.
\newblock Technical report, Rice University, USA, 2013.

\bibitem{BraKyr:11}
J..K Bradley, A.~Kyrola, D.~Bickson, and C.~Guestrin.
\newblock Parallel coordinate descernt for $l_1$-regularized loss minimization.
\newblock {\em ICML}, 2011.

\bibitem{TakBij:13}
M.~Takac, A.~Bijral, P.~Richtarik, and N.~Srebro.
\newblock Mini-batch primal and dual methods for svms.
\newblock Technical report, Univ. Edinburgh, March 2013.

\bibitem{RamNed:09}
S.~Sundhar Ram, A.~Nedic, and V.V. Veeravalli.
\newblock Incremental stochastic subgradient algorithms for convex
  optimization.
\newblock {\em SIAM Journal on Optimization}, 20(2):691--717, 2009.

\bibitem{NiuRec:12}
F.~Niu, B.~Recht, C.~Re, and S.~Wright.
\newblock Hogwild!: A lock-free approach to parallelizing stochastic gradient
  descent.
\newblock {\em NIPS}, 2012.

\bibitem{WanBer:13}
M.~Wang and D.~P. Bertsekas.
\newblock Incremental constraint projection-proximal methods for nonsmooth
  convex optimization.
\newblock Technical report, MIT, July 2013.

\bibitem{Nes:04}
Y.~Nesterov.
\newblock {\em Introductory Lectures on Convex Optimization: A Basic Course}.
\newblock Kluwer, Boston, USA, 2004.

\bibitem{RyaSup:10}
S.~Ryali, K.~Supekar, D.~A. Abrams, and V.~Menone.
\newblock Sparse logistic regression for whole-brain classiﬁcation of fmri
  data.
\newblock {\em NeuroImage}, 51(2):752--764, 2010.

\bibitem{YuaCha:10}
G.X. Yuan, K.W. Chang, C.J. Hsieh, and C.J. Lin.
\newblock A comparison of optimization methods and software for large-scale
  $l_1$-regularized linear classification.
\newblock {\em Journal of Machine Learning Research}, 11:3183--3234, 2010.

\bibitem{JamPau:13}
G.M. James, C.~Paulson, and P.~Rusmevichientong.
\newblock The constrained lasso.
\newblock Technical report, University of Southern California, 2013.

\bibitem{CheNg:12}
X.~Chen, M.~K. Ng, and C.~Zhang.
\newblock Non-lipschitz $\ell_p$ -regularization and box constrained model for
  image restoration.
\newblock {\em IEEE Transactions on Image Processing}, 21(12):4709--4721, 2012.

\bibitem{RocWet:98}
R.T. Rockafellar and R.J. Wets.
\newblock {\em Variational Analysis}.
\newblock Springer-Verlag, 1998.

\bibitem{WanLin:13}
P.W. Wang and C.J. Lin.
\newblock Iteration complexity of feasible descent methods for convex
  optimization.
\newblock Technical report, National Taiwan University, 2013.

\bibitem{MaZha:13}
S.~Ma and S.~Zhang.
\newblock An extragradient-based alternating direction method for convex
  minimization.
\newblock Technical report, Chinese University of Hong Kong, January 2013.

\bibitem{Man:85}
O.L. Mangasarian.
\newblock Computable numerical bounds for lagrange multipliers of stationary
  points of non-convex differentiable non-linear programs.
\newblock {\em Operations Research Letters}, 4(2):47--48, 1985.

\bibitem{Rob:73}
S.~M. Robinson.
\newblock Bounds for error in the solution set of a perturbed linear program.
\newblock {\em Linear Algebra and its Applications}, 6:69--81, 1973.

\bibitem{NecNed:13}
I.~Necoara and V.~Nedelcu.
\newblock Rate analysis of inexact dual first order methods: application to
  dual decomposition.
\newblock {\em IEEE Transactions on Automatic Control}, 59(5):1232--1243, 2014.

\bibitem{LevLew:08}
D.~Leventhal and A.S. Lewis.
\newblock Randomized methods for linear constraints: Convergence rates and
  conditioning.
\newblock Technical report, Cornell University, 2008.
\newblock http://arxiv.org/abs/0806.3015.

\end{thebibliography}
\thispagestyle{empty} \pagestyle{empty}

\end{document}